\documentclass[12pt,reqno]{amsart}
\usepackage{amssymb}

\newcommand{\e}{\varepsilon}

\renewcommand{\L}{\mathcal{L}}
\newcommand{\la}{\lambda}
\newcommand{\al}{\alpha}

\newcommand{\fy}{\varphi}

\newcommand{\p}{\partial}
\newcommand{\I}{\infty}

\newcommand{\ti}{\widetilde}
\newcommand{\R}{\mathbb{R}}
\newcommand{\C}{\mathbb{C}}
\newcommand{\N}{\mathbb{N}}
\newcommand{\Z}{\mathbb{Z}}

\newcommand{\F}{\mathcal{F}}
\renewcommand{\S}{\mathcal{S}}
\renewcommand{\Re}{\mathop{\mathrm{Re}}}
\renewcommand{\Im}{\mathop{\mathrm{Im}}}
\renewcommand{\bar}{\overline}
\renewcommand{\hat}{\widehat}

\numberwithin{equation}{section}

\newtheorem{theorem}{Theorem}[section]
\newtheorem{corollary}[theorem]{Corollary}
\newtheorem{conjecture}[theorem]{Conjecture}
\newtheorem{lemma}[theorem]{Lemma}

\theoremstyle{remark}
\newtheorem{remark}[theorem]{Remark}

\newcommand{\ran}{\rangle}
\newcommand{\lan}{\langle}

\newcommand{\weak}[1]{{\text{w-}#1}}
\newcommand{\lec}{{\ \lesssim \ }}
\newcommand{\gec}{{\ \gtrsim \ }}

\newcommand{\sgec}{\gtrsim}
\newcommand{\EQ}[1]{\begin{equation} \begin{split} #1
 \end{split} \end{equation}}

\newcommand{\Del}[1]{}

\newcommand{\pt}{&}
\newcommand{\pr}{\\ &}
\newcommand{\pq}{\quad}
\newcommand{\pn}{}
\newcommand{\pmt}{\\ &\qquad\times}
\newcommand{\prq}{\\ &\quad}

\newcommand{\LR}[1]{{\lan #1 \ran}}
\newcommand{\de}{\delta}

\newcommand{\be}{\beta}
\newcommand{\ka}{\kappa}
\newcommand{\ga}{\gamma}
\newcommand{\x}{\xi}
\newcommand{\y}{\eta}
\newcommand{\z}{\zeta}
\newcommand{\s}{\sigma}

\newcommand{\na}{\nabla}
\renewcommand{\th}{\theta}

\newcommand{\supp}{\operatorname{supp}}

\newcommand{\E}{\mathcal{E}}
\newcommand{\De}{\Delta}
\newcommand{\Om}{\Omega}
\newcommand{\Br}[1]{\LR{#1}}

\newcommand{\PX}[2]{\big\langle #1 \big| #2 \big\rangle_x}
\newcommand{\IN}[1]{\text{ in }#1}

\newcommand{\zn}{Z}
\newcommand{\zo}{z}
 
\newcommand{\J}{\check}
\newcommand{\Jp}[1]{{\J{#1}}^\pm}

\newcommand{\BB}{\mathcal{B}}
\newcommand{\NN}{\mathcal{N}_\zn}
\newcommand{\NO}{\mathcal{N}_\zo}
\newcommand{\Nv}{\mathcal{N}_v}
\newcommand{\Nu}{\mathcal{N}_u}
\newcommand{\naxy}{\na_\x^{(\y)}}
\newcommand{\naxz}{\na_\x^{(\z)}}

\begin{document}
\title[Scattering for Gross-Pitaevskii in 3D]{Scattering theory for \\ the Gross-Pitaevskii equation \\ in three dimensions}
\author{Stephen Gustafson,\quad Kenji Nakanishi,\quad Tai-Peng Tsai}
\begin{abstract}
We study global behavior of small solutions of the Gross-Pitaevskii equation in three dimensions. 
We prove that disturbances from the constant equilibrium with small, localized energy, disperse for large time, according to the linearized equation. 
Translated to the defocusing nonlinear Schr\"odinger equation, this implies asymptotic stability of all plane wave solutions for such disturbances. 
We also prove that every linearized solution with finite energy has a nonlinear solution which is asymptotic to it. 
The key ingredients are: (1) some quadratic transforms of the
solutions, which effectively linearize the nonlinear energy space, (2)
a bilinear Fourier multiplier estimate, which allows irregular
denominators due to a degenerate non-resonance property of the quadratic interactions, and (3) geometric investigation of the degeneracy in the Fourier space to minimize its influence. 
\end{abstract}
\maketitle
\tableofcontents

\section{Introduction}
We continue the study \cite{vac,vac2} of the global dispersive nature of solutions for the Gross-Pitaevskii equation (GP)
\EQ{ \label{GP}
 \pt i\psi_t + \De \psi = (|\psi|^2-1)\psi,\quad \psi:\R^{1+3}\to\C}
with the boundary condition $\lim_{|x|\to\I}\psi=1$. The equation itself is equivalent to the defocusing nonlinear Schr\"odinger equation (NLS) by putting $\fy=e^{-it}\psi$ 
\EQ{ \label{NLS}
 \pt i\fy_t + \De \fy = |\fy|^2\fy,}
but the nonzero boundary condition has nontrivial and often remarkable
effects on the space-time global behavior of the solutions both in dispersive and non-dispersive regimes, on which there has been extensive studies \cite{BGS,BOS,BS1,BSm,Ch,CJ1,CJ2,Ga,Ge,Gr2,Gr3,vac,vac2,J2,LaSch,LX,OS,Sp}. 
This boundary condition, or more generally $\lim_{|x|\to\I}|\psi|=1$, is natural in  some physical contexts such as superfluids and nonlinear optics, or generally in the hydrodynamic interpretation of NLS, where $|\psi|^2$ is the fluid density and the zeros correspond to vortices. 
(For more physical backgrounds, see \cite{BS2,FS,JR,JPR,SS} and the references therein.)

Hence it is not surprising that our GP equation, after some transformations \eqref{eq uzo}, is very similar to a version of the Boussinesq equation
\EQ{ \label{Bsq}
 u_{tt} - 2\De u + \De^2 u = -\De(u^2),}
which is a model for water waves and other fluid dynamics. More precisely, this equation can be regarded as a simplified version of our transformed equations (See Remark \ref{transBsq} for the details). 
Our analysis will reveal that the linear part, which is exactly the same for \eqref{Bsq} and our GP, has better dispersive properties than both the low frequency limit (the wave equation) and the hight frequency limit (the Schr\"odinger equation) for this type of bilinear interaction. 

We recall some known facts about GP \eqref{GP}. For any solution $\psi=1+u$, we have conservation of renormalized energy:
\EQ{
 E_1(\psi) \pt:= \int_{\R^3} |\na\psi|^2 + \frac{(|\psi|^2-1)^2}{2} dx 
 = \int_{\R^3} |\na u|^2 + \frac{(2\Re u+|u|^2)^2}{2} dx,}
provided it is initially finite, and then the solution $\psi$ is unique and global \cite{Ge}. As a remarkable feature of GP, there is a family of traveling wave solutions \cite{BS1,Ch} of the form $\psi(t,x)=v_c(x-ct)$ with finite energy for $0<|c|<\sqrt{2}$, where $\sqrt{2}$ is the sound speed. 
They are quite different from the solitary waves of the focusing NLS in their slow (algebraic) decay at the spatial infinity \cite{Gr3} and bounded (subsonic) range of speeds \cite{Gr2}. 

Recently it is proved \cite{BGS} that there is a lower bound on the energy of all possible traveling waves for \eqref{GP} in three dimensions:
\EQ{ \label{def E0}
 \E_0 := \inf\{E_1(\psi) \mid \text{$\psi(t,x)=v(x-ct)$ solves \eqref{GP} for some $c$}\}>0.}
If one believes that there is no more stable structure supported by \eqref{GP}, 
it is natural to expect that the regime below the threshold $\E_0$ is dominated by dispersion, as was conjectured in \cite{BS2}. 
Understanding the effect of dispersion in the nonlinearity is an essential step towards investigating asymptotic stability of traveling waves. 
However the problem is not quite easy, because we have quadratic interactions of the perturbation $u=\psi-1$, without any decaying factor, due to the nonzero constant background:  
\EQ{ \label{eq u0}
 iu_t + \De u - 2\Re u &= u^2+2|u|^2+|u|^2u.}
They are much stronger, for dispersive waves, than cubic interactions and quadratic ones with decaying factors, which typically arise in the stability analysis of solitary waves for NLS. 
In fact, if one considers the NLS with general quadratic terms as a model equation,  
\EQ{
 iu_t + \De u = \la_+ u^2 + \la_0 |u|^2 + \la_- \bar{u}^2,}
then the asymptotic behavior for $t\to\I$ of solutions from small localized initial data is known \cite{HN,HMN} only if $\la_0=0$ in three dimensions. 
From the technical view point, the quadratic power in three dimensions corresponds to the so-called Strauss exponent, because the $L^p$ decay estimate
\EQ{
 \|e^{it\De}\fy\|_{L^3(\R^3)} \lec |t|^{-1/2}\|\fy\|_{L^{3/2}(\R^3)},}
implies that $L^3$ is mapped back to the dual $L^{3/2}$ by the quadratic nonlinearity, with the critical decay order $|t|^{-1}$ for integrability.  
Thus one cannot get closed nonlinear estimates just from the $L^p$ decay, unless one starts with given final states (\cite{vac2} gives such a result for GP). 
Hence we have to take account of the oscillatory property of the quadratic terms. 
Then the term $|u|^2$ is worse than the others, because its phase by the linear approximation is stationary both in space and time at zero frequency $\x=0$. 

Actually, we cannot neglect the first order term $2\Re u$ for large
time behavior, so we should not expect too strong an analogy with the quadratic NLS. By the diagonalizing transform 
\EQ{
 u=u_1+iu_2 \mapsto v=u_1+iUu_2, \pq U=\sqrt{-\De(2-\De)^{-1}},}
we get the equation for $v$ with a self-adjoint operator $H=\sqrt{-\De(2-\De)}$ : 
\EQ{
 iv_t - H v = U(3u_1^2+u_2^2+|u|^2u_1) + i(2u_1u_2+|u|^2u_2).}
The equation for $U^{-1}v$ is\footnote{Here we naturally put $U^{-1}$ on $v$ to map the linear energy space $H^1(\R^3)$ for $v$ onto that for the wave equation, $\dot H^1(\R^3)$.} similar to the quadratic wave equation around $\x\to 0$
\EQ{ \label{NLW}
 u_{tt} - \De u = |\na u|^2,}
if we pick up the first nonlinearity $U(u_1^2)$ only. This equation is known \cite{John} to blow up from arbitrarily small initial data in $C_0^\I$. 
Note that the other terms $U(u_2^2)+2i(u_1u_2)$ are even worse, enhancing the low frequency through $u_2=U^{-1}\Im v$. 
The wave equation behaves worse in the quadratic interactions than the
Schr\"odinger (and the Klein-Gordon) equation because, beside the
slower $L^p$ decay of $e^{it\sqrt{-\De}}$, the waves propagating in
parallel directions are strongly resonant in the bilinear forms, unless they have the null structure (special coefficients killing exactly the parallel interactions). 

Despite of all these observations, we can prove that the solutions for GP with small localized initial energy disperse linearly at time infinity. 
Below we denote $\LR{x}=\sqrt{2+|x|^2}$ and $\LR{\na}=\sqrt{2-\De}$.  
$H^1$ is the standard Sobolev space with the norm $\|v\|_{H^1}=\|\Br{\na}v\|_{L^2}$, and $\LR{x}^{-1}H^1$ is the weighted space with the norm $\|\LR{x}v\|_{H^1}$. 
\begin{theorem} \label{thm:init}
There exists $\de>0$ such that: For any $u(0)\in H^1(\R^3)$ satisfying 
\EQ{ \label{init u}
 \int_{\R^3} \LR{x}^2(|\Re u(0,x)|^2 + |\na u(0,x)|^2) dx < \de^2,}
we have a unique global solution $\psi=1+u$ of \eqref{GP} such that 
$v := \Re u + i U\Im u$ 
satisfies $e^{itH}v\in C(\R;\LR{x}^{-1}H^1(\R^3))$ and for some $v_+\in\LR{x}^{-1}H^1(\R^3)$,   
\EQ{ \label{scatt}
 \pt\left\|v(t) - e^{-itH}v_+\right\|_{H^1} \le O(t^{-1/2}),
 \pq\left\|\LR{x}\left\{e^{itH}v(t) - v_+\right\}\right\|_{H^1} \to 0,}
as $t\to\I$. Moreover, we have $E_1(\psi)=\|\Br{\na}v_+\|_{L^2}^2$, and the correspondence $v(0)\mapsto v_+$ defines a bi-Lipschitz map between $0$-neighborhoods of $\LR{x}^{-1}H^1(\R^3)$. 
\end{theorem}
Actually we have some decay $t^{-\e}$ also for the weighted norm in \eqref{scatt}, but we will not try to specify it, since it is quite small. For pointwise decay, we can derive from the weighted energy estimate together with the $L^p$ estimate on $e^{-itH}$ that
\EQ{ \label{pt dec u}
 \|u_1(t)\|_{L^\I(\R^3)} \le O(t^{-1}), 
 \pq \|u_2(t)\|_{L^\I(\R^3)} \le O(t^{-9/10}).}

In view of the scaling property, the optimal (weakest) weight should be $\LR{x}^{-1/2}$ instead of $\LR{x}^{-1}$. 
The latter is however more convenient to estimate in the Fourier space, where it turns into one derivative. 
We can see that all the traveling waves with finite energy have \eqref{init u} finite, from their asymptotic behavior \cite{Gr3} as $|x|\to\I$. 
For that purpose, we should not use a weight stronger than $\LR{x}^{-3/2}$, hence  $\LR{x}^{-1}$ is the unique integer choice. 

By using fractional derivatives in the Fourier space, we could reduce
the weight slightly, but probably not to the optimal rate $1/2$, since
our argument needs to sacrifice part of the decay in some bilinear estimates with Fourier singularities. 

Those singularities arise because the linearized operator $e^{-itH}$ behaves like the wave equation around $\x\to 0$, where the parallel interactions become stronger. 
Since we know that solutions of nonlinear wave equations such as \eqref{NLW} blow up, we must use the difference between $H$ and $\sqrt{-2\De}$ in estimating the bilinear terms, which gives us some degenerate non-resonance property. 
The singular multipliers appear when we integrate on the phase in the Fourier regions for such interactions. 
In order to treat them, we will derive an estimate for bilinear Fourier multipliers (Lemma \ref{sbil}), which allows much less regular multipliers than the standard Coifman-Meyer type estimates. 
It will also be crucial to investigate carefully the geometric structure of the Fourier regions for those singularities, in order to exploit their smallness in volume.

On the other hand, we also have to take care of those bad quadratic terms which contain $u_2=U^{-1}\Im v$. The remedy was already given in \cite{vac,vac2}, by the quadratic transform
\EQ{ \label{def M}
 u \mapsto M(u) := v + \Br{\na}^{-2}|u|^2,}
which effectively removes those terms from the equation.  
We will introduce a new quadratic transform, which behaves better in
terms of regularity for high-high interactions. 
It turns out that the wave-like behavior of $e^{-itH}$ becomes an advantage when estimating the quadratic parts of those transforms, because it gives more decay for $\x\to 0$ than $e^{it\De}$. Indeed that is crucial to get the scattering in the same topology $\LR{x}^{-1}H^1$ as the initial data space. 

Extending the above result to the energy space is a supercritical problem from the scaling view point, and so it seems beyond our current technology (it is critical in four dimensions, where we have the small energy scattering \cite{vac}). However, we can still solve the final state problem without uniqueness. Since the energy space is essentially nonlinear for $u$, we need the  explicit linearizing map $M(u)$ to state our scattering result. 
\begin{theorem} \label{thm:fin}
For any $z_+\in H^1(\R^3)$, there exists a global solution $\psi=1+u$ of \eqref{GP} satisfying $E_1(\psi)=\|\Br{\na}z_+\|_{L^2(\R^3)}^2$ and as $t\to\I$ 
\EQ{ \label{scat z}
 \pt \|M(u) - e^{-itH}z_+\|_{H^1(\R^3)} \to 0.}
\end{theorem}
The proof depends on a compactness argument, so we get almost no information about the initial data set for the scattering.   
However, the above initial data result Theorem \ref{thm:init} implies
that it contains at least those data satisfying \eqref{init u}.

In \eqref{scat z}, we can replace $M(u)$ with $v$ if the norm is also replaced with $\dot H^1$, but the quadratic part is generally not even bounded in $L^2$ (see Remark \ref{L2 unbd} for the details). 
Hence the nonlinear effect is not really vanishing for large time in the energy sense, but instead the map $M$ essentially linearizes the (renormalized) energy space
\EQ{
 \pt F_1 := \{f \in \dot H^1(\R^3) \mid 2\Re f +|f|^2 \in L^2(\R^3)\},}
where $\dot H^1(\R^3)=\{f\in L^6(\R^3)\mid \na f\in L^2(\R^3)\}$ is the homogeneous Sobolev space. 
$F_1$ was introduced in \cite{Ge} with the distance $\de(f,g)$ defined by\footnote{The distance in \cite{Ge} is slightly different but equivalent to the above.} 
\EQ{
 \de(f,g)^2 = \|\na(f-g)\|_{L^2}^2 + \frac{1}{2}\||f|^2+2f_1-|g|^2-2g_1\|_{L^2}^2.}
More precisely, the energy space in three dimensions was characterized as 
\EQ{
 \{e^{i\th}(1+f)\mid \th\in\R,\ f \in F_1\}.}
As for the mapping property of $M$ on the energy space, we have the following. 
Let for any $L>0$ and $\ka>0$, 
\EQ{
 \pt F_1(L,\ka) := \{f\in F_1\mid E_1(1+f)\le L^2,\ \|f\|_{L^6}\le \ka\},
 \pr H^1(L,\ka) := \{f\in H^1\mid \|f\|_{H^1}\le L,\ \|f\|_{L^6}\le \ka\}.}
\begin{theorem} \label{thm:map}
(1) For any $f,g\in F_1$, we have 
\EQ{ \label{energy mapping}
 \de(f,g)^2 =  \|\Br{\na}(M(f)-M(g))\|_{L^2}^2 + \frac{1}{2}\|U(|f|^2-|g|^2)\|_{L^2}^2,}
hence the map $u\mapsto M(u)$ is Lipschitz continuous from $F_1$ to $H^1$. \par
(2) For any $L>0$, there exist $\ka>0$, $\ka'>2\ka$, and a unique map $R:H^1(L,\ka)\to F_1(2L,\ka')$ satisfying $M\circ R=id$. Moreover we have 
\EQ{
 \de(R(f),R(g)) \le 2\|\Br{\na}(f-g)\|_{L^2},}
namely the inverse is also Lipschitz in this domain. 
\end{theorem}
Remark that the smallness is required only in $L^6$, not in the energy, and the $L^6$ norm decays if the solution scatters. Hence we can apply the invertible maps to scattering solutions and their asymptotic profiles for sufficiently large time.  
Thus we can rewrite the scattering statement \eqref{scat z} as 
\EQ{
 \de(u(t),R(e^{-itH}z_+))\to 0 \pq (t\to\I).}

We can also translate the above results to {\it asymptotically plane wave solutions} for the cubic NLS \eqref{NLS}. 
First, it is easy to see that for any $a\ge 0$ and any $d\in\N$, every weak solution $\fy\in L^3_{loc}(\R^{1+d})$ satisfying $|\fy|=a$ on $\R^{1+d}$ is written in the form
\EQ{ \label{plane}
 \fy(t,x) = a e^{-i(a^2+|b|^2)t+ibx+ic},}
for some $b\in\R^d$ and $c\in\R$. 
The above solutions for \eqref{GP} can be regarded as perturbations of the plane wave $\fy=e^{-it}(1+u)$, and perturbations of the other plane waves are generated by the invariance of NLS under the scaling and Galilean transforms. 
More explicitly, for any solution $\psi=1+u$ of \eqref{GP}, for any
$a>0$, $b \in \R^3$ and $c\in\R$, we have a solution $\fy$ for \eqref{NLS} of the form 
\EQ{ \label{asy plane}
 \fy(t,x) = ae^{-i(a^2+|b|^2)t+ibx+ic}(1+u(a^2t,a(x-2bt))).}
In particular, Theorem \ref{thm:init} implies the asymptotic stability of all plane waves \eqref{plane} in the weighted energy norm \eqref{scatt}, as well as in $L^\I$ \eqref{pt dec u}, for perturbations $u(0)$ satisfying \eqref{init u}. 
Similarly, Theorem \ref{thm:fin} implies that there exist plenty of asymptotically plane wave solutions with finite (renormalized) energy, at least one for each asymptotic profile from $H^1$, $a>0$, $b\in\R^3$ and $c\in\R$.  

Notice that the situation is quite different for the standard boundary condition $a=0$.  Indeed the invariant scaling $\fy\mapsto \la\fy(\la^2 t,\la x)$ also changes the energy 
\EQ{
 E_0(\fy) = \int_{\R^3} |\na\fy|^2 + \frac{|\fy|^4}{2} dx,}
homogeneously, so the smallness of the energy $E_0(\fy)$ has no impact on the scattering, and the transform $M$ is reduced to the identity. 
Moreover, it is easy to observe that the free propagator $e^{it\De}$ is continuous on the energy space $\dot H^1\cap L^4$, but not globally bounded.  
Hence some free solutions from this space cannot be asymptotic to any nonlinear solution in the full energy. 
Once we restrict the solutions to $L^2(\R^3)$, then the scattering problem was completely solved by \cite{GV}, extending the earlier result \cite{LS} in a smaller space. But without $L^2$ finiteness, it is not even clear whether we can have a result similar to Theorem \ref{thm:fin} in $\dot H^1$, or we should instead modify the set of asymptotic profiles. 

We conclude the introduction by reinforcing the conjecture in \cite{BS2} into a precise statement of nonlinear scattering. Let $\E_0>0$ be as in \eqref{def E0}. 
\begin{conjecture} \label{conj}
For any global solution $\psi\in C(\R;1+F_1)$ of \eqref{GP} satisfying $E_1(\psi)<\E_0$, there is a unique $z_+\in H^1(\R^3)$ satisfying $E_1(\psi)=\|\Br{\na}z_+\|_{L^2}^2$ and 
\EQ{
 \|M(u(t))-e^{-itH}z_+\|_{H^1(\R^3)} \to 0 \pq (t\to\I).}
Moreover, the map $u(0)\mapsto z_+$ is a homeomorphism between the open balls of radius $\E_0^{1/2}$ around $0$ in $F_1$ and $H^1$. 
\end{conjecture}

The rest of this paper is organized as follows. Sections 2--3 are preparatory. Section 4 deals with the final data problem. Sections 5--13 are devoted to the initial data problem. In the appendix, we give a correction to our previous paper \cite{vac2}. 

In more detail: in Section 2, we collect some notation and basic estimates used throughout the paper. In Section 3, we compute the normal forms including the previous and the new ones. 
In Section 4, we prove the final data result Theorem \ref{thm:fin} and also Theorem \ref{thm:map}. 
In Section 5, we start with the initial data problem by setting up the
main function space $X$ and deriving its decay properties. In Section
6, we derive the weighted estimates for the normal form transform.
Section 7 deals with the initial data. 
In Section 8, we estimate the Strichartz norms without weight. In Section 9, we treat the easier parts of the weighted estimate, namely those without the phase derivative, and also the quartic terms. In Section 10, we treat the main terms, namely the bilinear forms with the phase derivative, using a bilinear multiplier estimate allowing singularities, and assuming some bounds on the singular bilinear multipliers. Those bounds are proved in Section 11. The cubic terms with the phase derivative are estimated in Section 12. Finally we prove Theorem \ref{thm:init} in Section 13. 

\section{Preliminaries}
In this section we establish some notation and basic setting for our problem. 
First, for function spaces, we denote the Lebesgue, the Lorentz, the Sobolev and the Besov spaces by $L^p$, $L^{p,q}$, $H^{s,p}$ and $B^s_{p,q}$ respectively, for $1\le p,q\le \I$ and $s\in\R$. 
The homogeneous versions of the latter two are denoted by $\dot H^{s,p}$ and $\dot B^s_{p,q}$. 
We do not use the spaces with $0\le p<1$, but instead we use the following convention
\EQ{
 L^{1/p} = L^p,\pq L^{1/p,1/q} = L^{p,q},\pq H^{s,1/p} = H^{s,p},}
for convenience of explicit computation of the H\"older exponents. 
For any Banach space $B$, we denote the same space with the weak topology by $\weak{B}$. 

We denote the Fourier transform on $\R^d$ by
\EQ{
 \pt \F\fy = \ti\fy(\x) := \int_{\R^d} \fy(x)e^{-ix\x} dx,
 \pr \F_x^\x [f(x,y)] = (\F_x^\x f)(\x,y):= \int_{\R^d} f(x,y)e^{-ix\x} dx,}
and the Fourier multiplier of any function $\fy$ by
\EQ{
 \pt\fy(-i\na)f := \F^{-1}[\fy(\x)\ti f(\x)],
 \pr\fy(-i\na)_x f(x,y) := (\F_x^\x)^{-1}[\fy(\x)\F_x^\x[f(x,y)]].}

Next we introduce a Littlewood-Paley decomposition (homogeneous version) in the standard way. Fix a cut-off function $\chi\in C_0^\I(\R)$ satisfying $\chi(x)=1$ for $|x|\le 1$ and $\chi(x)=0$ for $|x|\ge 2$. For each $k\in 2^\Z$, we define 
\EQ{
 \chi^k(x) := \chi(|x|/k) - \chi(2|x|/k),}
so that we have $\chi^k\in C_0^\I(\R^d)$, 
\EQ{
 \supp\chi^k \subset\{k/2<|x|<2k\}, \pq \sum_{k\in 2^\Z} \chi^k(x) = 1\pq(x\not=0).}
Hence the Littlewood-Paley decomposition is given by
\EQ{ \label{LP decop}
  f = \sum_{k\in 2^\Z} \chi^k(\na) f.}
We denote the decomposition into the lower and higher frequency by
\EQ{ \label{freq split}
  f_{<k} := \sum_{j<k} \chi^j(\na) f, 
 \pq f_{\ge k} := \sum_{j\ge k} \chi^j(\na) f.}

For any function $B(\x_1,\dots,\x_N)$ on $(\R^d)^N$, we associate the $N$-multilinear operator $B[f_1,\dots f_N]$ defined by 
\EQ{
 \pt \F_x^\x B[f_1,\dots f_N] 
 \pn = \int_{\x=\x_1+\cdots+\x_N} B(\x_1\dots\x_N)\ti f_1(\x_1)\cdots\ti{f_N}(\x_N) d\x_2\cdots d\x_N,}
which is called a multilinear Fourier multiplier with symbol $B$, thus we identify the symbol and the operator. 
Whenever we write a symbol in the variables $(\x,\x_1,\dots\x_N)$, it should be understood as a funciton of $(\x_1,\dots\x_N)$ by substitution $\x=\x_1+\dots+\x_N$. Hence for example, a product of single multipliers can be written as
\EQ{
 B_0(\x)B_1(\x_1)B_2(\x_2)[f,g] = B_0(-i\na)\left\{(B_1(-i\na)f)\cdot(B_2(-i\na)g)\right\}.}

In addition, we introduce the following convention for bilinear multipliers. 
For any bilinear symbol $B(\x_1,\x_2)$, 
we assign the variables $\y$ and $\z$ such that $\y=\x_1$ and $\z=\x_2$, but regarding $(\x,\y)$ and $(\x,\z)$ respectively as the independent variables. 
Hence the partial derivatives of the symbol $B$ in each coordinates are given by
\EQ{
 \pt (\naxy B,\na_\y B) = (\na_{\x_2} B(\y,\x-\y), (\na_{\x_1}-\na_{\x_2})B(\y,\x-\y)),
 \pr (\naxz B,\na_\z B) = (\na_{\x_1} B(\x-\z,\z), (\na_{\x_2}-\na_{\x_1})B(\x-\z,\z)).}

For any number or vector $a$, we denote
\EQ{ \label{def UHetc}
 \pt \Br{a}:=\sqrt{2+|a|^2},\pq \hat{a}:=\frac{a}{|a|},
 \pq U(a):=\frac{|a|}{\Br{a}},\pq H(a):=|a|\Br{a},}
which will mostly appear in the Fourier spaces, in particular 
$U= \sqrt{-\De/(2-\De)}$ and $H= \sqrt{-\De(2-\De)}$. 
For any complex-valued function $f$, we often denote the complex conjugate by 
\EQ{
 f^+ := f, \pq f^-:=\bar{f},}
to treat them in a symmetric way. Also we denote  
\EQ{ \label{def J}
 J f = e^{-itH}xe^{itH}f, \pq \J f = \F e^{itH} f,\pq \Jp{f} = \F e^{\pm itH} f^\pm.}
Remark that these operations are time dependent. 

Finally, we collect a few basic estimates, which will be used throughout the paper. 
\begin{lemma}
Let $2\le p\le\I$, $0\le\th\le 1$, $s\in\R$, and $\s=1/2-1/p$. Then we have
\EQ{ \label{Bp dec H}
 \|e^{-itH}\fy\|_{\dot B^s_{p,2}} \lec |t|^{-(d-\th)\s}\|U^{(d-2+3\th)\s}\LR{\na}^{2\th\s}\fy\|_{\dot B^s_{p',2}},}
where $p'=p/(p-1)$ is the H\"older conjugate. For $2\le p<\I$, we have also
\EQ{ \label{Lp dec H}
 \|e^{-itH}\fy\|_{L^{p,2}} \lec |t|^{-(d-\th)\s}\|U^{(d-2+3\th)\s}\LR{\na}^{2\th\s}\fy\|_{L^{p',2}}.}
\end{lemma}
The case $\th=0$ is similar to the Schr\"odinger equation, but we gain $U^{(d-2)\s}$ around $\x=0$.  
The case $\th=1$ with $|\x|<1$ is the same as the wave equation, where we have worse decay for $t\to\I$, but better decay for $\x\to 0$, than the case $\th=0$. 
This reflects the fact that the group velocity for $e^{-itH}$ is bounded from below $|\na H(\x)|>\sqrt{2}$. For $|\x|\gec 1$, the estimate with $\th>0$ follows from interpolation between $\th=0$ and the Sobolev embedding, hence it is just the same as for Schr\"odinger.  
\begin{proof}
The estimate in the Besov space is proved by the standard stationary phase argument after the Littlewood-Paley decomposition, and its interpolation with $L^2$ conservation. See \cite{vac} for the details in the case $\th=0$. 
Then we can replace Besov with $L^p$ spaces by the embedding 
\EQ{
 \dot B^0_{p,2} \subset L^p, \pq \dot B^0_{p',2} \supset L^{p'},}
for $2\le p<\I$. The Lorentz version follows from real interpolation. 
\end{proof}
From the above decay estimate, we get the Strichartz estimate (cf. \cite{vac}), 
\EQ{ \label{Strz} 
 \|e^{-itH}f\|_{L^\I H^1 \cap U^{1/6}L^2 B^1_{6,2}} \lec \|f\|_{U^{-1/6}L^2 B^1_{6/5,2} + L^1 H^1},}
where we gain $U^{1/6}$ for the endpoint norm.  

For bilinear Fourier multipliers, we will use the Coifman-Meyer bilinear multiplier estimate \cite{CM} (see also \cite{GK}): if
\EQ{ \label{CMcond}
 |\p_{\x_1}^\al\p_{\x_2}^\be B(\x_1,\x_2)| \lec (|\x_1|+|\x_2|)^{-|\al|-|\be|},}
for up to sufficiently large $\al,\be$, then for any $p_j\in(1,\I)$ satisfying $1/p_0=1/p_1+1/p_2$, we have 
\EQ{ \label{bilCM}
 \|B[f,g]\|_{L^{p_0}(\R^d)} \lec \|f\|_{L^{p_1}(\R^d)} \|g\|_{L^{p_2}(\R^d)}.}
As was shown in \cite{GK}, one cannot generally replace the right hand side of \eqref{CMcond} by $|\x_1|^{-|\al|}|\x_2|^{-|\be|}$, but we can reduce our regular multipliers to the above case. 
However the above estimate is not applicable when the multipliers are really singular due to divisors coming from the phase integrations in the main estimates for Theorem \ref{thm:init}. 
In this case, we will use another bilinear estimate (Lemma \ref{sbil}), which allows general multipliers with singularities. 
As for regular multipliers, the following specific form will be often sufficient:    
\begin{lemma}
Let $d\in\N$ and $k\in\N$. Then we have  
\EQ{  \label{bil nonsing}
 \pt \sup_{0\le a\le 1}\left\|\frac{\Br{\x_1}^{2k(1-a)}\Br{\x_2}^{2ka}}{\Br{(\x_1,\x_2)}^{2k}}[f,g]\right\|_{L^{p_0}_x(\R^d)}
 \lec \|f\|_{L^{p_1}_x(\R^d)} \|g\|_{L^{p_2}_x(\R^d)},}
for any $p_0,p_1,p_2\in (1,\I)$ satisfying $1/p_0=1/p_1+1/p_2$, where $(\x_1,\x_2)$ denotes the $2d$ dimensional vector and so 
\EQ{
 \Br{(\x_1,\x_2)} = \sqrt{2+|\x_1|^2+|\x_2|^2}.}
\end{lemma}
\begin{proof}
Denote the above operators by $B_a$, which is analytic on the strip $0\le\Re a\le 1$ as bilinear operators $(H^d)^2\to L^2$. When $a=0$ or $a=1$, $B_a$ satisfies \eqref{CMcond}, hence the estimate follows from Coifman-Meyer. Since the Fourier multiplier $\LR{t}^{-1-d}\LR{\x}^{it}$ is bounded on $L^p$ uniformly for $t\in\R$, we get
\EQ{
 \sup_{t\in\R} \LR{t}^{1+d}(\|B_{it}[f,g]\|_{L^{p_0}} + \|B_{1+it}[f,g]\|_{L^{p_0}}) \lec \|f\|_{L^{p_1}} \|g\|_{L^{p_2}},}
which is extended to the whole strip $0\le\Re a\le 1$ by the three lines theorem for analytic functions. 
\end{proof}

\section{Normal forms} \label{s:normal}
The idea of using normal forms for asymptotic behavior of nonlinear dispersive equations goes back to \cite{Sh} in the case of nonlinear Klein-Gordon equations. 
For the NLS, quadratic terms have in general resonance sets with one codimension (except for $\bar{u}^2$), which give rise to singularities in the normal forms. 
Here the resonance set is given in the spatial frequency where the time oscillation of the nonlinearity coincides with the linear evolution. We call it temporal resonance. 

In \cite{HMN} for the scattering of NLS in three dimensions, such
singularity was avoided by applying the normal form to remove only
some of the quadratic terms, leaving the others with better decay properties. 
Here the decay for the remainders comes from the fact that the oscillation phase of the nonlinearity relative to the linear evolution is not stationary at those spatial frequencies. We call it spatial resonance. 
When we simply say ``non-resonant", it means that it is {\it not simultaneously} temporally and spatially resonant. 

Our transforms follow the same idea as \cite{HMN}. However, the actual
forms and computations are totally different from the NLS, due to the
lower order term $2\Re u$ and the special structure of the
nonlinearity. In addition, we can not completely avoid singularities
because of some degeneracy in the non-resonance property, for which we need another idea (a bilinear estimate for irregular multipliers). 

In the previous paper \cite{vac2}, we introduced the normal form $M(u)$ as defined in \eqref{def M}. 
As is indicated in Theorem \ref{thm:map}, this transform is quite natural in view of the nonlinear energy, and will be essential for the proof of Theorem \ref{thm:fin}. 
However, it does not fit the weighted estimate for Theorem \ref{thm:init} very well, where it suffers from a derivative loss. 
To avoid assuming higher regularity, we introduce another normal form which is more suited to the weighted estimate. 
For that purpose, we start with general bilinear transforms. 

First we rewrite the equation for $u=\psi-1=u_1+iu_2$ in terms of the real variables:
\EQ{ \label{eq u}
 \pt \dot u_1 = -\De u_2 + 2 u_1u_2 + |u|^2 u_2,
 \pr \dot u_2 = -(2-\De)u_1 - 3u_1^2 - u_2^2 - |u|^2u_1.}
We consider a nonlinear transform given in the form\footnote{We will introduce $B_j$ and $C_j$ afterward by transforming $B_j'$ and $C_j'$.} 
\EQ{
 w := u+B_1'[u_1,u_1]+B_2'[u_2,u_2]}
where $B_j'$ are real-valued symmetric bilinear Fourier multipliers, which we will  specify later. From the above equation for $u$, we derive
\EQ{
 \pt (i\p_t+\De-2\Re)B_1'[u_1,u_1]
 \pr= -(2-\De)B_1'[u_1,u_1] + 2iB_1'[u_1,-\De u_2+2u_1u_2 + |u|^2u_2],
 \pr (i\p_t+\De-2\Re)B_2'[u_2,u_2]
 \pr= -(2-\De)B_2'[u_2,u_2] + 2iB_2'[u_2,(\De-2)u_1-3u_1^2 -u_2^2 -|u|^2u_1],}
and hence we get the equation for $w$:
\EQ{
 \pn(i\p_t + \De - 2\Re)w
 \pt= B'_3[u_1,u_1] + B'_4[u_2,u_2]
 \pn + iB'_5[u_1,u_2]
 \pr + |u|^2u_1
 \pn + iC'_3[u_1,u_1,u_2] + iC'_4[u_2,u_2,u_2]
 \pn  + iQ_1(u),}
where $B_j'$ ($j=3,4,5$) are bilinear multipliers defined by,  
\EQ{
 \pt B'_3 = 3-\Br{\x}^2B_1',
 \pq B'_4 = 1-\Br{\x}^2B_2',
 \pq B'_5 = 2+2|\x_2|^2B_1'-2\Br{\x_1}^2B_2',}
$C_j'$ ($j=3,4$) are cubic multipliers defined by
\EQ{ \label{def C'}
 \pt C'_3(\x_1,\x_2,\x_3) = 1 + 4B_1'(\x_1,\x_2+\x_3) - 6B_2'(\x_1+\x_2,\x_3), 
 \pr C'_4(\x_1,\x_2,\x_3) = 1 - 2 B_2'(\x_1+\x_2,\x_3),}
and $Q_1(u)=  2 B_1'[u_1,|u|^2u_2]- 2B_2'[u_2,|u|^2u_1]$ gathers the quartic terms. 

Now we choose our $B_1'$ and $B_2'$. Recall the diagonalized equation for $v=u_1 + i U u_2$: 
\EQ{ \label{eq v}
 iv_t - Hv = \Nv(u) := U(3u_1^2+u_2^2+|u|^2u_1) + i(2u_1u_2+|u|^2u_2),}
where we observe that we have $U^{-1}$ in each $u_2$ compared with $v$, and also on the imaginary part relative to the real part. 
To counteract this effect, we want  
\EQ{
 \pt B'_4 = O(|\x_1||\x_2|), \pq B'_5 = O(|\x||\x_2|), \pq C_4'=O(|\x_1|+|\x_2|+|\x_3|).}
It is easy to see that $B_4'=B_5'=0$ is impossible due to the symmetry of $B_1'$ and $B_2'$. 
The simplest choice for $B_4'=0$ is given by
\EQ{ \label{old def B12}
 \pt B_1'=B_2'=\Br{\x}^{-2},
 \pq B_4' = 0, \pq B_5' = 4\Br{\x}^{-2}\x\x_2, 
 \pq C_3'=C_4'=U^2, }
where $\x\x_2=\x\cdot\x_2$ denotes the inner product in $\R^3$. 
This leads to $M(u)$ in \eqref{def M}, and is the most natural from the energy viewpoint \eqref{energy mapping}. Moreover, the cubic part has a subtle but remarkable non-resonance property, which will be revealed in Section \ref{s:fin} for the proof of Theorem \ref{thm:fin}.  

On the other hand, the simplest choice for $B_5'=0$ is given by  
\EQ{ \label{def B12}
 \pt -B_1' = B_2' = \Br{(\x_1,\x_2)}^{-2} = \frac{1}{2+|\x_1|^2+|\x_2|^2},
 \pr B'_4 = \frac{-2\x_1\x_2}{2+|\x_1|^2+|\x_2|^2},
 \pq B'_5 = 0, 
 \pq C'_4 = \frac{|\x_1+\x_2|^2+|\x_3|^2}{2+|\x_1+\x_2|^2+|\x_3|^2},}
where we omit $C_3'$ since it is just complicated without any important structure. 
This choice is not directly linked to the energy, and the cancellation in the cubic part is weaker. Its advantage is that we do not lose a derivative for  $\LR{\x}\ll|\x_1|\sim|\x_2|$ as in $B_5'$ of \eqref{old def B12}. 
Such a derivative loss for high-high interactions can be saved by increasing the regularity.  
However, in the weighted estimate for bilinear terms, we will encounter another derivative loss \eqref{B^X bound2} by integration on the phase, and therefore we would need $H^2_x$ to close the estimates with the choice \eqref{old def B12}. 

The latter derivative loss can be observed in the Schr\"odinger case as well: 
look at the weighted estimate for the typical bilinear term $|e^{-it\De}\fy|^2$ in the Fourier space  
\EQ{
 \F_x^\x xe^{it\De}|e^{-it\De}\fy|^2 = i\na_\x\int e^{it\Om} \bar{\F\fy}(-\y)\F\fy(\x-\y) d\y,}
where $\Om := |\x|^2 - |\y|^2 + |\x-\y|^2$. 
Partial integration on $e^{it\Om}$ gives a factor of the form
\EQ{
 \frac{|\naxy\Om|}{|\na_\y\Om|} = \frac{|2\x-\y|}{|\x|},}
which causes a derivative loss in the region $|\x|\ll|\y|\sim|\y-\x|$. 
However, those two normal forms eventually lead to the same asymptotic behavior, in the case of weighted energy scattering. See Section \ref{ss:other normal}. 

We remark that there are other possible choices. For example, if we want decay in $B_3'$ as well, we may choose
\EQ{
 \pt B_1'=\Br{\x}^{-2}(3-4\Br{\x}^{-2}\x_1\x_2),
 \pq B_2' = \Br{\x}^{-2}(1-4\Br{\x}^{-2}\x_1\x_2),
 \pr B_3'=B_4'=4\Br{\x}^{-2}\x_1\x_2,
 \pr B_5'=4\Br{\x}^{-4}(4\x_0\x_2+4|\x_0|^2|\x_2|^2+2(\x_1\x_2)[\x_0(\x_0+4\x_1-4\x_2)],}
but it suffers from derivative loss in all $B_j'$, so we will not use it. 
In the rest of the paper, those notations $B_j'$, $C_j'$ and $Q_1$ are reserved for the second choice \eqref{def B12}.

Now we fix the notation for our transforms corresponding to \eqref{old def B12} and \eqref{def B12} respectively by\footnote{The pair $(v,z)$ in the previous paper \cite{vac2} is equal to $U^{-1}(v,\zo)$ in this paper.}
\EQ{ \label{w to z}
 \pt \zo := M(u) = v + \Br{\x}^{-2}|u|^2,
 \pr \zn := v + b(u):= v-\Br{(\x_1,\x_2)}^{-2}[u_1,u_1] + \Br{(\x_1,\x_2)}^{-2}[u_2,u_2].}
Then we get coupled equations for $(v,\zo)$ and $(v,\zn)$ in the form 
\EQ{ \label{eq uzo}
 \pt v = \zo - \Br{\x}^{-2}|u|^2, 
 \pq i\dot \zo - H\zo = \NO(u),}
\EQ{ \label{eq uzn}
 \pt v = \zn - b(u), 
 \pq i\dot \zn - H\zn = \NN(v).}
The nonlinear terms are given by 
\EQ{ \label{def NO}
 \NO(u) := U\left\{2u_1^2+|u|^2u_1\right\}
 \pn - i\Br{\na}^{-2}\na\cdot\left\{4u_1\na u_2+\na(|u|^2u_2)\right\},}
\EQ{ \label{def NN}
 \NN(v) \pt:= B_3[v_1,v_1] + B_4[v_2,v_2]
 \pn + C_1[v_1,v_1,v_1] +C_2[v_2,v_2,v_1]
 \prq + i C_3[v_1,v_1,v_2]  +i C_4[v_2,v_2,v_2]
 \pn + iQ_1(u),}
with the bilinear multipliers $B_j$ defined by using $B_j'$ in the case \eqref{def B12}, 
\EQ{
 \pt B_3 = U(\x)B'_3 = \frac{-2U(\x)(4+4|\x_1|^2+4|\x_2|^2-\x_1\x_2)}{2+|\x_1|^2+|\x_2|^2},
 \pr B_4 = U(\x)U(\x_1)^{-1}U(\x_2)^{-1}B_4' = \frac{-2U(\x)\Br{\x_1}\Br{\x_2}\hat{\x_1}\hat{\x_2}}{2+|\x_1|^2+|\x_2|^2},}
and the cubic multipliers $C_j$ defined by using $C_j'$ in \eqref{def C'}, 
\EQ{
 \pt C_1 = U(\x),
 \pq C_2 = U(\x)U(\x_1)^{-1}U(\x_2)^{-1},
 \pr C_3 = U(\x_3)^{-1}C'_3,
 \pq C_4 = U(\x_1)^{-1}U(\x_2)^{-1}U(\x_3)^{-1}C'_4,}
and the quartic term $Q_1$ is given by
\EQ{
 Q_1(u) = -2\Br{(\x_1,\x_2)}^{-2}[u_1,|u|^2u_2] - 2\Br{(\x_1,\x_2)}^{-2}[u_2,|u|^2u_1].}
We have the following bounds:
\EQ{ \label{bound on B}
 \pt|B_3|+|B_4| \lec U(\x),
 \pr|C_j| \lec U(\x_1)^{-1}U(\x_2)^{-1} + U(\x_2)^{-1}U(\x_3)^{-1} +
 U(\x_3)^{-1}U(\x_1)^{-1}.}
More precisely, they are written as linear combinations of
products of those terms on the r.h.s. of~\eqref{bound on B}
and regular multipliers as in \eqref{bil nonsing},
together with the Riesz operator. 
\begin{remark} \label{transBsq}
By the transform $u\mapsto v:=u-iH^{-1}\dot u$, the Boussinesq equation \eqref{Bsq} can be rewritten as 
\EQ{
 iv_t + H v = U(v_1^2),}
hence it can be regarded as a simplified version of our equation for $(v,\zo)$. In fact, our proof for Theorem \ref{thm:init} directly applies to this equation, while Theorem \ref{thm:fin} and Theorem \ref{thm:map} rely much more on the nonlinear energy structure, which is different for \eqref{Bsq}.  
For the (physically relevant) two dimensional case, refer to \cite{vac2} for a final data result, which is also applicable to this equation. 
\end{remark}

\section{Final data problem}\label{s:fin}
In this section, we consider the final state problem in the energy space, proving Theorems \ref{thm:fin} and \ref{thm:map}. 
Our strategy for Theorem \ref{thm:fin} is essentially the same as \cite{asshort} for NLS, using compactness to get weak convergence, and energy conservation to make it strong. 
The normal form transform $M(u)$ is essential both for the uniform bound and for the equicontinuity of approximating sequence of solutions. 
In addition, we will use a hidden non-resonance property in the cubic terms of \eqref{eq uzo} to get the equicontinuity. 
It is noteworthy that the argument in \cite{asshort} does not work for NLS without $L^2$ finiteness, whereas for GP we are using only the energy conservation. 
We start with the proof of Theorem \ref{thm:map}, since the result will be used in the proof of Theorem \ref{thm:fin}. 

\begin{proof}[Proof of Theorem \ref{thm:map}]
The identity \eqref{energy mapping} is straightforward. 
We construct $R$ by solving the following equation for $g$ with any given $f\in H^1(L,\ka)$:
\EQ{
 g_1=R_f(g_1) := f_1-\Br{\na}^{-2}\left\{g_1^2+(U^{-1}f_2)^2\right\}, \pq g_2=U^{-1}f_2.}
We use the following product estimate for general $u,v$: 
\EQ{ \label{prod in Rf}
 \|\Br{\na}^{-2}uv\|_{\dot H^1} \pt\lec \|U(uv)\|_{L^2} \sim \|\na(uv)\|_{H^{-1}}
 \pr\lec \|v\|_{\dot H^1}\|\na u\|_{H^{-1,3}} + \|u\|_{L^\I+L^6}\|\na v\|_{L^2}
 \pr\lec \|v\|_{\dot H^1}\left\{\|Uu\|_{L^3} + \|u\|_{L^6}^{1/2}\|Uu\|_{L^6}^{1/2} \right\}
 \pr\lec \|v\|_{\dot H^1}\|u\|_{\dot H^1}^{1/2}\|Uu\|_{L^6}^{1/2}.}
Applying it to $R_f(g_1)$, we get 
\EQ{
 \pt\|\Br{\na}^{-2}g_2^2\|_{\dot H^1} 
  \lec \|Ug_2^2\|_{L^2} \lec \|f\|_{L^6}^{1/2}\|f\|_{H^1}^{3/2},
 \pr\|R_f(u)-R_f(v)\|_{\dot H^1}
  \lec\|u+v\|_{\dot H^1}^{1/2}\|u+v\|_{L^6}^{1/2}\|u-v\|_{\dot H^1},}
and in addition,
\EQ{
  \pt\|\Br{\na}^{-2}g_1^2\|_{\dot H^1} \lec \|Ug_1^2\|_{L^2} \lec \|g_1\na g_1\|_{L^{3/2}}
   \lec \|g_1\|_{L^6}\|g_1\|_{\dot H^1},
  \pr\|\Br{\na}^{-2}g_1^2\|_{L^6} \lec \|g_1^2\|_{L^3} = \|g_1\|_{L^6}^2. }
Hence the map $g\mapsto R_f(g_1)+ig_2$ is a contraction on the set 
\EQ{
 \{g\in F_1(2L,C(\ka+L^{3/2}\ka^{1/2})) \mid g_2=U^{-1}f_2\},} 
for the $\dot H^1$ norm and some $C>0$, if $\ka>0$ is sufficiently small. 
Hence it has a unique fixed point $g_1=R_f(g_1)$, by which we can define $R(f)=g_1+iU^{-1}f_2$. Then we have $M\circ R=id$ and the uniqueness of $R$ follows from that for the fixed point. 
For the continuity of $R$, let $f,g\in H^1(L,\ka)$ and $u=R(f)$, $v=R(g)$. Then we have by \eqref{prod in Rf}, 
\EQ{
 \|U(|u|^2-|v|^2)\|_{L^2}
 \pt \lec \|u+v\|_{\dot H^1}^{1/2}\|u+v\|_{L^6}^{1/2}\|u-v\|_{\dot H^1}
 \ll \|u-v\|_{\dot H^1},}
if $\ka$ is small. Moreover,
\EQ{
 \|u-v\|_{\dot H^1}^2 = \|u_1-v_1\|_{\dot H^1}^2 + \|f_2-g_2\|_{H^1}^2.}
For the real component, we have
\EQ{
 \pt \|u_1-v_1\|_{\dot H^1}
 =\|R_f(u_1)-R_g(v_1)\|_{\dot H^1}
 \pr \le \|f_1-g_1\|_{\dot H^1} + \|\Br{\na}^{-2}(u_1^2-v_1^2)\|_{\dot H^1}
  + \|\Br{\na}^{-2}(U^{-1}f_2)^2-(U^{-1}g_2)^2\|_{\dot H^1}.}
where the second term is bounded by
\EQ{
 \|u_1+v_1\|_{L^6} \|u_1-v_1\|_{\dot H^1} \ll \|u_1-v_1\|_{\dot H^1},}
and the third term is bounded by using \eqref{prod in Rf}, 
\EQ{
 \|f_2-g_2\|_{H^1}\|f_2+g_2\|_{H^1}^{1/2}\|f_2+g_2\|_{L^6}^{1/2}
 \ll \|f_2-g_2\|_{H^1}.}
Hence for any $\e>0$, we can get by choosing $\ka>0$ small enough,  
\EQ{
 \|u_1-v_1\|_{\dot H^1}
 \le \|f_1-g_1\|_{H^1} + \e\|f_2-g_2\|_{H^1},}
and the desired Lipschitz condition follows.
\end{proof}

\begin{proof}[Proof of Theorem \ref{thm:fin}]
First we rewrite the equation \eqref{eq uzo} for $\zo=M(u)$, using the renormalized charge density $q(u):=|\psi|^2-1=2u_1+|u|^2$:
\EQ{ \label{eqz}
 \pt i\zo_t - H \zo = \NO^1(u)+\NO^2(u),
 \prq \NO^1 := \frac{U}{2}(q^2 - |u|^4 - 2u_1^3) 
  \pn- \frac{i\na}{\Br{\na}^2}\cdot\left\{2q\na u_2+u_1(2u_2\na u_1-u_1\na u_2)\right\},
 \prq \NO^2 := -U(u_2^2u_1) - i\Br{\na}^{-2}\na\cdot(u_2^2\na u_2).} 
Let $L=2\|\Br{\na}z_+\|_{L^2}+1$ and $z^0:=e^{-itH}z_+$. Since $\|z^0(t)\|_{L^6_x}\to 0$ as $t\to\I$, for large $T>0$ we have $z^0(T)\in H^1(L,\ka)$. Let $\psi^T$ be the global solution of \eqref{GP} satisfying
\EQ{
 \psi^T = 1 + u^T,\pq u^T(T) = R(z^0(T)),}
and let $\zo^T=M(u^T)$. Then we have $\zo^T(T)=z^0(T)$ and the above equation \eqref{eqz} with $(\zo,u)=(\zo^T,u^T)$ and $q=q^T=q(u^T)$. Moreover we have
\EQ{
 \|\Br{\na}\zo^T\|_{L^2} \le E_1(\psi^T)^{1/2} = \de(R(z^0(T)),0)\le 2\|\Br{\na}z_+\|_{L^2},}
so $\zo^T$, $u^T$ and $q^T$ are uniformly bounded in $H^1$, $F_1$ and $L^2$ respectively. 
By the Sobolev embedding, $u^T$ is bounded in $L^6_x$. Combined with the $L^2$ bound on $q^T$ and the fact that $(L^3+L^6)\cap L^2\subset L^3$, it implies also that $u_1^T$ and $q^T$ are bounded in $L^3_x$. 

Moreover, $u^T$ and $\zo^T$ are equicontinuous in $C(\R;\S'(\R^3))$ by their equations. 
Hence for some sequence $T\to\I$, $u^T$ and $\zo^T$ converge in $C(\R;\weak{\dot H^1})$ and $C(\R;\weak{H^1})$ respectively. 
We denote their limits by $u^\I$ and $z^\I$. 
Then $q^T$ also converges to $q^\I=q(u^\I)$ in $C(\R;L^2_{loc})$ and $C(\R;\weak{L^2})$. 
Also it is easy to check that $v^\I:=u_1^\I+iUu_2^\I$ and $z^\I$ satisfy the equation \eqref{eq uzo}. 

To get weak convergence of $e^{itH}\zo^T$ globally in time, we need its equicontinuity at $t=\I$, i.e., for any test function $\fy\in\S(\R^3)$ we have
\EQ{
 \int_{t_1}^{t_2} \PX{\fy}{e^{itH}\NO(u^T(t))} dt \to 0}
as $t_2>t_1\to\I$ uniformly in $T$. Those terms in $\NO^1(u)$ are bounded respectively by (here we omit $T$), 
\EQ{
 \pt\|q^2\|_{L^{12/11}} \lec \|q\|_{L^2\cap L^3}^2,
 \pr\|U|u|^4\|_{L^{12/11}} \lec \|u^3\na u\|_{L^1} \lec \|u\|_{L^6}^3\|\na u\|_{L^2},
 \pr\|U(u_1^3)\|_{L^{12/11}} \lec \|u_1^2\na u_1\|_{L^1} \lec \|u_1\|_{L^3\cap L^6}^2\|\na u\|_{L^2}.
 \pr\|q\na u_2\|_{L^{12/11}} \lec \|q\|_{L^2\cap L^3}\|\na u\|_{L^2},
 \pr\|u_1u\na u\|_{L^{12/11}} \lec \|u_1\|_{L^3\cap L^6}\|u\|_{L^6}\|\na u\|_{L^2},}
hence by using the $L^{12}$ decay of $e^{itH}$ \eqref{Lp dec H}, 
\EQ{
 \int_{t_1}^{t_2} |\PX{\fy}{e^{itH}\NO^1(u^T(t))}| dt \lec \int_{t_1}^{t_2} t^{-5/4} dt \lec t_1^{-1/4} \to 0.}
The remaining term $\NO^2$ is bounded only in $L^{6/5}_x$, which gives by the $L^p$ decay the critical decay rate $1/t$, and so we cannot get compactness in this way. 
More precisely, when we decompose them into dyadic frequencies by \eqref{LP decop} 
\EQ{
 \NO^2 = -\sum_{j,k\in 2^\Z} U\left\{\chi(\na)^j(u_2^2) \chi(\na)^k u_1\right\} + i\frac{\na}{\Br{\na}^2}\cdot\left\{\chi(\na)^j(u_2^2)\chi(\na)^k\na u_2\right\},}
we have trouble only when $j\ll\min(1,k)$, since otherwise $U$ or $\na$ outside gives the factor $j$ and so those terms are bounded in $L^{12/11}_x$ as above. 
In order to bound the remaining part, we use a non-resonance property due to its  special structure:
\EQ{ \label{decop NO1}
 \NO^2 = -U(u_2^2\bar{v}) - i\Br{\na}^{-2}\left\{H(\x)U(\x_2)-\x\x_2\right\}[u_2^2,u_2],}
where the second bilinear multiplier has the small factor $\x_1$, since it equals 
\EQ{
 \pt \Br{\x}^{-2}\left\{|\x|(\Br{\x}U(\x_2)-|\x|)+|\x|^2-\x\x_2\right\}
 \pr = \left\{ \frac{|\x||\x_2|}{(\Br{\x}+\Br{\x_2})\Br{\x_2}}
 - \frac{|\x|}{|\x|+|\x_2|} \right\}\frac{(\x_1+2\x_2)\cdot\x_1}{\Br{\x}^2} + \frac{\x\x_1}{\Br{\x}^2},}
and the multiplier inside the big braces is bounded by the Coifman-Meyer estimate \eqref{bilCM} if restricted to $|\x_1|\ll|\x_2|\sim|\x|$. So this term is bounded in $L^{12/11}_x$. 

For the first term in \eqref{decop NO1}, we integrate on $e^{it(H(\x)+H(\x_2))}$ 
\EQ{ \label{IBP for fin}
 \pt\int_{t_1}^{t_2} e^{it(H(\x)+H(\x_2))}U[u_2^2,\bar{e^{itH}v}] dt
 \pr= [-ie^{itH}\BB[u_2^2,\bar{v}]]_{t_1}^{t_2}
  + \int_{t_1}^{t_2} ie^{itH}\BB[2u_2\dot u_2,\bar{v}] - e^{itH}\BB[u_2^2,\bar{\Nv(u)}] dt,}
with the bilinear multiplier $\BB$ defined by 
\EQ{
 \BB = (H(\x)+H(\x_2))^{-1}U(\x) = \Br{\x}^{-2}\left\{1+H(\x_2)H(\x)^{-1}\right\}^{-1},}
where the multiplier in the braces is bounded by Coifman-Meyer \eqref{bilCM} if $|\x_1|\ll|\x_2|\sim|\x|$. From the equation \eqref{eq u}, we have
\EQ{
 \pt \dot u_2 = -q + \De u_1 - q u_1 \in L^\I_t H^{-1}_x,
 \pr \Nv(u) = U(|u|^2+q u_1) + iqu_2 \in L^\I_t(L^{3/2}\cap L^2)_x,}
so $u_2\dot u_2\in \dot H^1\times H^{-1}$ is bounded in $H^{-1,3/2}\cap H^{-2}$. 
Hence \eqref{IBP for fin} restricted to $|\x_1|\ll\min(1,|\x_2|)$ is bounded in 
$t_1^{-1/2}L^3_x + t_1^{-1/4} L^{12}_x$. 

\newcommand{\K}{\acute}
Thus we conclude that ${\K\zo}^T:=e^{itH}\zo^T$ is equicontinuous in $C([-\I,\I];\S'(\R^3))$, and so by choosing appropriate sequence $T\to\I$, we get 
\EQ{
 {\K\zo}^T \to {\K\zo}^\I \IN{C([-\I,\I];\weak{H^1})},} 
together with ${\K\zo}^\I(\I)=z_+$ from the initial condition $\zo^T(T)=e^{-iTH}z_+$. 
By the lower semi-continuity for $T\to\I$ (along the sequence), we have 
\EQ{
 E_1(\psi^\I) \pt= \|\Br{\na}\zo^\I\|_{L^2}^2 + \frac{1}{2}\|U|u^\I|^2\|_{L^2}^2
 \pn \le \liminf_{T\to\I} E_1(\psi^T(T))
 = \|\Br{\na}z_+\|_{L^2}^2,}
and for $t\to\I$,
\EQ{
 \|\Br{\na}z_+\|_{L^2}^2 \le \liminf_{t\to\I} E_1(\psi^\I(t)) = E_1(\psi^\I).}
Hence both inequalities must be equality, which implies that the convergence is strong in both cases. 
Hence $e^{itH}\zo^\I(t)\to z_+$ strongly in $H^1_x$ as $t\to\I$. 
\end{proof}
\begin{remark} \label{L2 unbd}
We can show that for some asymptotic profiles $z_+\in H^1$, the above constructed solution has $v_1$ unbounded in $L^2_x$ for large $t$ ($v_2$ is bounded in $L^2_x$ by the energy). 
So we really need the nonlinear map $M$ to have the scattering statement \eqref{scat z} in $H^1$. 
The proof goes by contradiction. 
If $v_1=u_1$ is uniformly bounded in $L^2$, by the $L^2$ bound of $q$, $|u|^2$ is also bounded. 
Then it is easy to see that $\NO(u)$ is bounded in $UL^p_x$ for $1\le p\le 6/5$. 
Thus we have by the $L^p$ decay with $1/p:=1/6-\e$ for $\e>0$ small,  
\EQ{
 \|e^{i(t-T)H}\zo(t) - \zo(T)\|_{U^{1+1/3+\e}L^p_x} 
 \pt= \|\int_T^t e^{isH}\NO(u(s))ds\|_{U^{4/3+\e}L^p_x} 
 \pr\lec \int_T^t s^{-1-3\e}\|\NO(u(s))\|_{UL^{5/6+\e}_x} ds \lec T^{-3\e},}
and so by letting $t\to\I$, we get
\EQ{
 e^{-iTH}z_+-\zo(T) \in U^{4/3+\e}L^p_x.}
Since $U^{-1}z_2(T)=u_2(T)\in L^\I_TL^4_x\subset L^\I_T\dot H^{-1/2-\e/3,p}$ for $T$ large, we have 
\EQ{
 (\Im e^{-iTH}z_+)_{<1} = (Uu_2(T) + U^{4/3+\e}L^\I_TL^p_x)_{<1}
 \subset U^{4/3}L^\I_TL^p_x,}
which is clearly not implied by the assumption $z_+\in H^1$. 
For example, let 
\EQ{
 z_+ = \sum_{k\in 4^{-\N}}ie^{iH/k}k^{3/2+3\e}f(k x),}
for some $f\in\S$ satisfying $\supp\ti f\subset\{1<|\x|<2\}$. Then $z_+\in H^1$ but
\EQ{
 \|U^{-4/3}e^{-iH/k}z_+\|_{\dot B^0_{p,\I}} \gec k^{-4/3+3/2+3\e-3/p}\|f\|_{L^p_x} \to \I \pq (1/k\to\I).}

However it is not clear whether $\|v_1(t)\|_{L^2_x}$ is growing, or infinite for all time ($v(t)\in L^2_x$ persists in time by the equation for $v$). We expect that both cases do happen, depending on the asymptotic data. 
\end{remark}

\section{Initial data problem}
The strategy for Theorem \ref{thm:init} largely follows that in \cite{HMN} for the NLS: 
we will derive an a priori bound on the weighted energy norm of the solution after removing the free propagator, namely $e^{itH}v(t)$. Then the a priori bound gives space and time decay properties of the solution by getting back the propagator, and this decay property is used in turn to bound the Duhamel integral for the a priori estimate. 
However, the linear decay property is not sufficient to bound the quadratic terms, so that we have to integrate by parts on the phase to gain more integrability and decay from the oscillations. At this point we meet the special difficulty of our equation, that is the singularity due to degeneration of the nonresonance property around $\x\to 0$. 
To treat it, we will derive a bilinear estimate for singular Fourier multipliers, and investigate the shape of the degeneration in a similar way as in our previous paper \cite{vac2} for the two dimensional case. 

Now we set up our function spaces. We will derive iterative estimates for both $v$ and $\zn$ of \eqref{eq uzn} in the following space-time norms. Recall the notation \eqref{def J} for $J$. 
\EQ{ \label{main est}
 \pt \|\zn(t)\|_{X(t)} := \|\zn(t)\|_{H^1_x} + \|J\zn(t)\|_{H^1_x},
 \pq \|\zn\|_X := \sup_t \|\zn(t)\|_{X(t)},
 \pr \|\zn\|_S := \|\zn\|_{L^\I_t H^1_x} + \|U^{-1/6}\zn\|_{L^2_t H^{1,6}_x}.}
The last norm is finite for $e^{-itH}\fy$ if $\fy\in H^1(\R^3)$, by the Strichartz estimate \eqref{Strz}. 
The key ingredient is the $H^1$ bound on $J\zn$, which can be written in the Fourier space
\EQ{
 \|J\zn\|_{H^1_x} \sim \|\LR{\x} \na_\x \J \zn\|_{L^2_\x}.}
More precisely, we are going to prove 
\EQ{
 \left\|\int_T^t e^{-i(t-s)H}\NN(s) ds\right\|_{X(T,\I)} \lec \LR{T}^{-\e},}
for some small $\e>0$. 
This estimate implies the desired scattering for $\zn$ in $\LR{x}^{-1}H^1$.  
The scattering for $v$ is the same, because we will prove that the difference $b(u)$ is vanishing faster in the same space $X$. 

\subsection{Decay property by the weighted norm}
We derive a few decay properties as $t\to\I$ and $\x\to 0$ of the $X$ space \eqref{main est}. First, the commutator relations $[\na_j,J_k] = \de_{j,k}$ and $[\Br{\na},J]=-\Br{\na}^{-1}\na$ imply  
\EQ{
 \|J \na v(t)\|_{L^2_x} + \|J\Br{\na}v(t)\|_{L^2_x} \lec \|Jv(t)\|_{H^1_x} + \|v(t)\|_{L^2_x}.}
Since $1/|x|\in L^{d,\I}(\R^d)$, we have by the Sobolev and the H\"older, 
\EQ{
 \|v(t)\|_{\dot H^{-1}_x} \pt\sim \|e^{itH}v(t)\|_{\dot H^{-1}_x}
 \pr\lec \|e^{itH}v(t)\|_{L^{6/5,2}_x} \lec \|xe^{itH}v(t)\|_{L^2_x}
 \sim \|Jv(t)\|_{L^2_x}.}
Thus we obtain 
\EQ{ \label{v L2 bounds}
 \|U^{-2}v\|_{L^6_x} 
 \pt \lec \|v\|_{\dot H^{-1}} + \|v\|_{H^1_x} \sim \|U^{-1}v\|_{H^1_x} 
 \pn \lec \|v(t)\|_{X(t)}.}

Choosing $\th=0$ and $p=6$ in the $L^p$ decay estimate \eqref{Lp dec H}, we have 
\EQ{
 \|\LR{\na}U^{-1/3}v(t)\|_{L^6_x}
  \pt\lec t^{-1}\|\LR{\na}e^{itH}v(t)\|_{L^{6/5,2}_x}
  \pr\lec t^{-1}\|J\Br{\na}v(t)\|_{L^2_x}
  \lec t^{-1}\|v(t)\|_{X(t)}.}
Combining with the Sobolev bound \eqref{v L2 bounds}, we get 
\EQ{ \label{v decay}
 \pt \||\na|^{-2+5\th/3} v_{<1}(t)\|_{L^6_x} \lec \min(1,t^{-\th})\|v(t)\|_{X(t)},
 \pr \||\na|^{\th} v_{\ge 1}(t)\|_{L^6_x} \lec \min(t^{-\th},t^{-1})\|v(t)\|_{X(t)},}
for $0\le\th\le 1$, where $v_{<1}$ and $v_{\ge 1}$ denote the smooth separation of frequency \eqref{freq split}. 
In particular, we have
\EQ{ \label{U-1vL6}
 \pt \|U^{-1}v(t)\|_{L^6_x} \lec \LR{t}^{-3/5}\|v(t)\|_{X(t)},
 \pq \|U^{-1}v\|_{L^2_t H^{1,6}_x} \lec \|v\|_{S\cap X}.}
Also we obtain the Strichartz bound on $u=v_1+iU^{-1}v_2$:
\EQ{ \label{Stz u}
 \pt \|u\|_{L^\I H^1} \lec \|U^{-1}v\|_{L^\I H^1} \lec \|v\|_X,
 \pr \|u\|_{L^2 H^{1,6}}
 \lec \|v_{\ge 1}\|_{L^2 H^{1,6}} + \|U^{-1}v_{<1}\|_{L^2 L^6} \lec \|v\|_{X\cap S}.}

For the estimate on $b(u)$, it is better to use the wave-type decay. Choosing $p=4$ and $\th=2/3$ in \eqref{Lp dec H}, and using complex interpolation, we have
\EQ{ \label{U-1vL4}
 \|\LR{\na}^{2/3}U^{-1}v(t)\|_{L^4_x}
  \pt\lec t^{-7/12}\|U^{-1/4}\LR{\na}e^{itH}v(t)\|_{L^{4/3,2}_x}
  \pr \lec t^{-7/12}\|\LR{\na}e^{itH}v(t)\|_{L^{6/5,2}_x}^{3/4}\|U^{-1}\LR{\na}e^{itH}v(t)\|_{L^2_x}^{1/4}
  \pr \lec t^{-7/12}\|v(t)\|_{X(t)}.}

\section{Estimates on the normal form}
In this section, we derive decay estimates on $b(u)$ (defined in \eqref{w to z}), and the invertibility of the mapping $v\mapsto \zn$. 
First from the $L^p$ decay property of $v$ and the bilinear estimate \eqref{bil nonsing}, we have 
\EQ{ \label{b bound}
 \pn\|b(u)\|_{H^{2,p}_x}
  \pt\lec \|U^{-1}v\|_{L^{p_1}_x} \|U^{-1}v\|_{L^{p_2}_x} \lec \|U^{-1}v\|_{L^2_x}^{2-\th} \|U^{-1}v\|_{L^6_x}^{\th} 
  \pr\lec \LR{t}^{-3\th/5}\|v(t)\|_{X(t)}^2,}
for $0<\th\le 2$, where
\EQ{
 \frac{1}{p} = 1 - \frac{\th}{3} = \frac{1}{p_1}+\frac{1}{p_2},
 \pq 2\le p_1,p_2\le 6.}
For the weighted norm, we compute $Jb(u)$ in the Fourier space, using the notation in \eqref{def J}. It is given by a linear combination of terms of the form 
\EQ{
 \pt e^{-itH(\x)}\na_\x \int_{\x=\y+\z} e^{itH(\x)\mp H(\y)\mp H(\z)}B(\y,\z)\Jp v(\y) \Jp v(\z) d\y
 \pr = e^{-itH(\x)}\int e^{it\Om}\Bigl[(\naxy B + it\naxy\Om\cdot B) \Jp v(\y) \Jp v(\z) + B \Jp v(\y) \na \Jp v(\z)\Bigr] d\y
 \pr = \F \left[(\naxy B+it\naxy\Om\cdot B)[v^\pm,v^\pm] + B[v^\pm,(Jv)^\pm]\right],}
with $B=\Br{(\x_1,\x_2)}^{-2}$ or $\Br{(\x_1,\x_2)}^{-2}U(\x_1)^{-1}U(\x_2)^{-1}$, and
\EQ{
 \pt \Om := H(\x)\mp H(\y)\mp H(\z), \pq \naxy\Om=\na H(\x)\mp\na H(\z).}
By the bilinear estimate \eqref{bil nonsing} and the $L^p$ decay property \eqref{v L2 bounds}, \eqref{U-1vL6} and \eqref{U-1vL4}, we have 
\EQ{ \label{est Jb} 
 \pt \|\naxy B[v^\pm,v^\pm]\|_{H^1_x}
  \lec \|U^{-1}v\|_{L^3_x} \|U^{-2}v\|_{L^6_x} \lec \LR{t}^{-3/10}\|v\|_X^2, 
 \pr \|t\naxy\Om\cdot B[v^\pm,v^\pm]\|_{H^1_x}
  \lec \|t^{1/2}U^{-1}v\|_{L^4_x}^2
   \lec \LR{t}^{-1/6}\|v\|_X^2,
 \pr \|B[v^\pm,(Jv)^\pm]\|_{H^1_x}
  \lec \|U^{-1}v\|_{L^3_x} \|U^{-1}Jv\|_{L^6_x}
    \lec \LR{t}^{-3/10}\|v\|_X^2.}
Thus we obtain
\EQ{ \label{dec bX}
 \pt\|b(u)(t)\|_{H^1_x} \lec \LR{t}^{-9/10}\|v\|_{X(t)}^2, 
 \pq \|Jb(u)(t)\|_{H^1_x} \lec \LR{t}^{-1/6}\|v\|_{X(t)}^2,
 \pr \|b(u)\|_{S(T_1,T_2)} \lec \LR{T_1}^{-7/10}\|v\|_{X(T_1,T_2)}^2,}
for any $t\ge 0$ and any $T_2\ge T_1\ge 0$. 
Hence by the Banach fixed point theorem, the map $v\mapsto v+b(u)$ is bi-Lipschitz in small balls around $0$ of $X(t)$ uniformly for all $t\in\R$, and also globally in $X\cap S$. 

\subsection{Other normal forms} \label{ss:other normal}
It is a natural question if we get the same asymptotics for other normal forms, in particular, for $\zo$ defined in \eqref{w to z}. 
It is easy to see that the above argument works to get
\EQ{
 \|\Br{\na}^{-2}|u|^2\|_{X(t)} \lec \LR{t}^{-1/6}\|v\|_X^2.}
For this, we need to modify the second inequality in \eqref{est Jb} for the interaction with $|\x|+1\ll|\x_1|\sim|\x_2|$ because we lose regularity. Using the Schr\"odinger type decay estimate in that region, we get nevertheless
\EQ{
 \|t\naxy\Om\cdot\Br{\x}^{-2}[u^\pm,u^\pm]\|_{H^1_x}
 \pt\lec \|t^{1/2}U^{-1}v\|_{L^4_x}^2 + \|tv_{>1}\|_{H^{1,6}_x}\|v_{>1}\|_{L^3_x}
 \pr\lec \LR{t}^{-1/6}\|v\|_{X}^2.}
Therefore we have
\EQ{
 \|\zo(t)-v(t)\|_{X(t)} \lec \LR{t}^{-1/6}\|v\|_X^2,}
in other words, $\zo$, $v$ and $\zn$ have the same asymptotics in the strong topology of $X$, as long as $v$ is in $X$.  

\section{Initial condition}
Here we check that the smallness condition \eqref{init u} on $u(0)$ is equivalent to smallness of $v(0)$ in $X(0)$. 

By Sobolev and H\"older in Lorentz spaces, \eqref{init u} implies 
\EQ{
 \|u(0)\|_{L^2_x} \lec \|\na u(0)\|_{L^{6/5,2}_x}
 \lec \|x\na u(0)\|_{L^2_x} \lec \de.}
Hence we have 
\EQ{
 \|u(0)\|_{H^1_x} + \|xu_1(0)\|_{L^2_x} + \|x\na u_1(0)\|_{L^2_x} 
  + \|x\na u_2(0)\|_{L^2_x} \lec \de.}
Moreover, by using the commutator $[x,U]=\F^{-1}[i\na_\x,U(\x)]\F = i\na_\x U$,   
\EQ{
 \|\LR{\na}x U u_2\|_{L^2_x}
 \pt\le \|\LR{\na}U x u_2\|_{L^2_x} + \|\LR{\na}(\na_\x U)u_2\|_{L^2_x}
 \pr\lec \|\na xu_2\|_{L^2_x} + \|u_2\|_{L^2_x},}
hence we get $\|\LR{x}v(0)\|_{H^1_x} \lec \de$. 

On the other hand, since 
\EQ{
 [J_j,\na_k U^{-1}] = -\de_{j,k}U^{-1} + \na_{x_k} i \na_{\x_j} U^{-1},}
we have
\EQ{ \label{com Ju}
 \|J\na u\|_{L^2_x}+\|\na Ju\|_{L^2_x}
 \pt\lec \|U^{-1}\na Jv\|_{L^2_x} + \|[J,\na U^{-1}]v\|_{L^2_x} + \|u\|_{L^2_x}
 \pr\lec \|Jv\|_{H^1_x} + \|U^{-1}v\|_{L^2_x} \lec \|v(t)\|_{X(t)}.}
In particular, by letting $t=0$, 
\EQ{
 \|\Br{x}u_1(0)\|_{L^2_x} + \|\Br{x}\na u(0)\|_{L^2_x}
  \lec \|\Br{x}v(0)\|_{H^1_x}.}

\section{Estimates without weight} \label{ss:est w/o w}
In this section, we estimate $\zn$ in the Strichartz norm $L^\I_t H^1_x \cap U^{1/6}L^2_t H^{1,6}_x$,  
by putting the nonlinear terms $\NN(v)$ in the dual Strichartz space, as in \eqref{Strz}. 

The bilinear terms are bounded by using \eqref{bil nonsing}:  
\EQ{ \label{Stz on B}
 \pt \|B_j[v^\pm,v^\pm]\|_{L^{4/3}_t H^{1,3/2}_x} 
  \lec \|v\|_{L^\I H^1} \|v\|_{L^{4/3} L^6} \lec \|v\|_{X\cap S}^2,}
for $j=3,4$, where the last factor is bounded by the decay estimate \eqref{v decay}. 
On interval $t\in(T,\I)$ for $T>1$, we can estimate instead
\EQ{
  \pn\|B_j[v^\pm,v^\pm]\|_{L^1_t H^1_x}
  \pt\lec \|v\|_{L^4_{t>T} H^{1,3}} \|v\|_{L^{4/3}_{t>T} L^6}
  \pr\lec T^{-1/2}\|v\|_{L^\I H^1}^{1/2} \|tv\|_{L^\I_{t>T} H^{1,6}_x}^{3/2}
  \pn\lec T^{-1/2} \|v\|_X^2.} 
 
The cubic terms $C_j$ with $j=1,2,3,4$ are similarly bounded by 
\EQ{ \label{Stz on C}
 \|C_j[v^\pm,v^\pm,v^\pm]\|_{L^2_t H^{1,6/5}_x}
 \pt\lec \|U^{-1}v\|_{L^\I H^1} \|U^{-1}v\|_{L^4 L^6}^2
 \pn\lec \|v\|_{X}^3,}
by using \eqref{v L2 bounds} and \eqref{v decay}, together with the regular bilinear estimate \eqref{bil nonsing}. For $t>T$ we gain at least $T^{2(-3/5+1/4)}=T^{-7/10}$ from the $L^6_x$ decay of $U^{-1}v$. 
The quartic term $Q_1$ is estimated by using \eqref{Stz u},
\EQ{
 \pt\|Q_1(u)\|_{L^{4/3} H^{1,3/2}}
 \pn\lec \|u\|_{L^4 H^{1,3}}^3 \|u\|_{L^\I L^6}
 \lec \|v\|_{X\cap S}^4.}
For $t>T$ we gain at least $T^{-3/5}$ from the $L^6_x$ decay of $u=v_1+iU^{-1}v_2$.

In conclusion, we have 
\EQ{
 \pn\left\|\int_T^t e^{-i(t-s)H}\NN ds\right\|_{S(T,\I)}
 \pt\lec \|\NN\|_{L^2_{t>T} H^{1,6/5} + L^1_{t>T} H^1} \pr\lec T^{-1/2}(\|v\|_{X\cap S}^2+\|v\|_{X\cap S}^4),}
so once we get a uniform bound on $\|v\|_{X\cap S}$, the scattering of $\zn$ and $v$ in $H^1$ follows: 
\EQ{ \label{scat z H1}
 \exists! v_+\in H^1,\pq \|\zn(t)-e^{-itH}v_+\|_{H^1} + \|v(t)-e^{-itH}v_+\|_{H^1} \lec \LR{t}^{-1/2},}
where we used \eqref{dec bX}.

\section{Estimates without phase derivative}
Now we proceed to the $L^\I_t H^1_x$ bound on $J\zn$. 
First we rewrite the equation for $\zn$ by replacing $v$ with $\zn$ in the bilinear terms:
\EQ{
 \pt \NN(v)=B_3[\zn_1,\zn_1] + B_4[\zn_2,\zn_2] 
  + \sum_{j=1}^5 C_j + Q_1(u) + Q_2(u) =: \NN'(v,\zn),
 \pr C_5(v,v,\zn):=-2B_3[b(u),\zn_1], \pq Q_2(u):=B_3[b(u),b(u)],}
because we are going to integrate on the phase in time (in some Fourier region), 
where the difference of oscillation between $\zn$ and $v$ will become essential. 
In the higher order terms, it is negligible thanks to the time decay of $b(u)$. 

Applying $J$ to $\NN'$, we get bilinear terms in the Fourier space like
\EQ{
  \int_0^t\int \naxy \left\{e^{is\Om}B_j(\y,\x-\y)\Jp{\zn}(s,\y) \Jp{\zn}(s,\x-\y)\right\} d\y ds,}
with $j=3,4$ and the phase $\Om := H(\x)\mp H(\y)\mp H(\x-\y)$. 
If the derivative $\naxy$ lands on $B_j$, then since $\naxy B_j$ is a bounded multiplier, its contribution is estimated as above. 
If $\naxy$ lands on $\Jp{\zn}$, we get terms of the form in the physical space 
\EQ{
 \int_0^t e^{isH} B_j[\zn^\pm,(J \zn)^\pm] ds,}
whose contribution in $H^1_x$ is bounded by using the Strichartz together with the regular bilinear estimate \eqref{bil nonsing}, 
\EQ{
 \|B_j[\zn^\pm,(J \zn)^\pm]\|_{L^{4/3}_t H^{1,3/2}_x}
 \pt\lec \|\zn\|_{L^{4/3} H^{1,6}} \|J \zn\|_{L^\I H^1}
 \pn\lec \|\zn\|_{X\cap S}^2, }
where the $L^{4/3} H^{1,6}$ norm was bounded by the Strichartz $S$ for small $t$ and by the decay \eqref{v decay} for large $t$. 
Hence if we restrict on $t\in(T,\I)$, then we get additional factor $T^{-1/4}$. 

Similarly, we have cubic terms like
\EQ{ \label{C phaseD} 
 \na_\x \int_0^t\iint_{\x=\x_1+\x_2+\x_3} e^{is\Om}C_j(\x_1,\x_2,\x_3)\Jp{v}(s,\x_1)\Jp{v}(s,\x_2)\Jp{v}(s,\x_3) d\x_1 d\x_2 ds,}
for $1\le j\le 4$, where $\Om=H(\x)\mp H(\x_1)\mp H(\x_2)\mp H(\x_3)$. If $\na_\x$ hits $C_j$, its contribution is bounded in the same way as before. 
If $\na_\x$ hits $\Jp{v}$, then its contribution is estimated by using the Strichartz and the bilinear estimate \eqref{bil nonsing}: 
\EQ{ \label{JC}
 \pn\|C_j[v^\pm,v^\pm,(J v)^\pm]\|_{L^2_t H^{1,6/5}_x}
 \pt\lec \|J v\|_{L^\I H^1} \|U^{-1}v\|_{L^2 H^{1,6}} \|U^{-1}v\|_{L^\I L^6}
 \prq+ \|U^{-1}J v\|_{L^\I L^6} \|U^{-1}v\|_{L^\I H^1} \|v\|_{L^2 H^{1,6}}
 \pr\lec \|v\|_{X\cap S}^3,}
where 
if the derivative in the $H^{1,6/5}_x$ norm lands on $Jv$ (with large frequency), it is dominated by the first term on the right, 
otherwise we use the second term. 
Refer to \eqref{U-1vL6} for the $L^2 H^{1,6}$ norm on $U^{-1}v$. 
For large $t>T$, we gain at least $T^{-1/2}$. 

For $C_5$, we have just to replace the last $\Jp{v}$ with $\Jp{\zn}$ in \eqref{C phaseD}, hence the final bound in \eqref{JC} is replaced by $\|v\|_{X\cap S}^2\|\zn\|_{X\cap S}$. 

The quartic terms $Q_j(u)$ with $j=1,2$ are regular. So we write in the physical space 
\EQ{
 J\int_0^t e^{-i(t-s)H}Q_j(u)ds
 =\int_0^t e^{-i(t-s)H}(x-s\na H(\x))Q_j(u) ds.}
Then by Strichartz, their contribution in $L^\I H^1$ is bounded by
\EQ{ \label{JN1,4}
  \|x Q_j(u)\|_{L^{2} H^{1,6/5}}
 + \|t Q_j(u)\|_{L^2 H^{2,6/5}}.}
For the first term we need to estimate $xu$. Since $Ju=xu-t(\na_\x H) u$, we have by using \eqref{com Ju},  
\EQ{
 \pt\|xu\|_{L^6_x+tL^2_x}
  \lec \|Ju\|_{\dot H^1_x} + \|(\na_\x H) u\|_{L^2_x}
  \lec \|v(t)\|_{X(t)}.}
Then we have 
\EQ{
 \|xQ_j(u)\|_{L^2 H^{1,6/5}}
 \pt\lec \|u\|_{L^\I L^6} \|u\|_{L^2 H^{1,6}} \|u\|_{L^\I L^3} \|u\|_{L^\I L^6}
 \prq + \|xu\|_{L^\I(L^6+tL^2)} \|\LR{t}^{1/2}u\|_{L^\I L^6}^2 \|u\|_{L^\I L^3\cap L^2 L^\I}
 \pr\lec \|v\|_{X\cap S}^4,}
and 
\EQ{
 \|tQ_j(u)\|_{L^2 H^{2,6/5}} 
  \lec \|t^{1/2}u\|_{L^\I L^6}^2 \|u\|_{L^2 L^6} \|u\|_{L^\I L^3}
 \lec \|v\|_{X\cap S}^4.}
For $t>T$, we gain at least $T^{-1/4}$ from the decay of $u$ in $L^\I L^6\cap L^\I L^3\cap L^2 L^\I$ by \eqref{v decay}. 

Thus it remains only to estimate the terms with the derivative $\na_\x$ landing on the phase $e^{is\Om}$ in $B_j$ and $C_j$. 
 

\section{Estimates with phase derivative in bilinear terms}
This and the next sections make the heart of this paper. 
If the derivative lands on the phase in the bilinear terms, then the pointwise decay as above is at best $t^{-1}L^3_x$, which is far from sufficient. 
So we should exploit the non-resonance property, through integration by parts in $\y$ and $s$.
We decompose $B_j$ with $j=3,4$ smoothly into two parts:
\EQ{ \label{ST reson decop}
 B_j = B_j^X + B_j^T,}
where $B_j^X$ is supported in $(\x,\y)$ where the interaction is spatially non-resonant, and $B_j^T$ in the temporally non-resonant region. 
Here non-resonance means simply that either $\y$ or $s$ derivative of $e^{is\Om}$ does not vanish. 

We integrate the phase in $\y$ for $B_j^X$ and in $s$ for $B_j^T$. The strict intersection of the spatially resonant and temporally resonant regions is only at $\x=0$, which can be compensated by the decay at $\x=0$ of the bilinear forms. 
This is the reason why the above decomposition is possible. 

The same type of argument has been used in \cite{vac2} for the 2D final data problem. The main difference from there is that now we need $L^p$-type bilinear estimates to close our argument for the initial data problem, 
and that we are not free to integrate by parts because the entry functions are the unknown solutions, whereas in the final data problem they were the given asymptotic profile. 

After the integration by parts, we get multipliers whose {\it derivatives} have much stronger singularity than allowed in the standard $L^p$ multiplier estimate, even if it were linear.  
The singularity is due to the behavior of $H(\x)$ around $\x\to 0$, which is close to the wave equation and thus enhancing resonance between the parallel interactions. 
Roughly speaking, the singularity increases for each derivative twice
as fast as for the standard multipliers, even with the best decomposition of $B^X+B^T$. 

To overcome this difficulty, we take advantage of the room in decay in the 3D case to reduce differentiation of the symbol, employing a bilinear estimate with loss in the H\"older exponent. 
It turns out that there is a narrow balance between the singularity and the loss in $L^p$, such that we can close all the estimates. 

Specifically, we use only $1/2+\e$ or $1+\e$ derivatives for small frequencies, while the standard $L^p$ estimates need at least $3/2+\e$ derivatives of the symbol. 
Moreover, it will be important for us to exploit the smallness of region where the singularity is the strongest.  

The plan of these two sections is as follows. 
First we prove the bilinear estimate allowing some singularities. 
Secondly we carry out the estimates on all the bilinear terms, assuming some bounds on the bilinear Fourier multipliers with divisors. 
Finally in the next section, we prove those bounds for each multiplier, using geometric properties of the resonance sets. 

\subsection{Singular bilinear multiplier with Strichartz}
We introduce mixed (semi-) norms $\L^p \dot B^s_{q,r}$ for the symbols by
\EQ{ \label{def LB}
 \|f(\x,\y)\|_{\L^p \dot B^s_{q,r}}
 := \|j^s \chi^j(\na)_\y f(\x,\y)\|_{\ell_j^r L^p_\x L^q_\y(2^\Z\times\R^d\times\R^d)},}
where $\chi^j$ is as in \eqref{LP decop}. 
\begin{lemma} \label{sbil}
Let $0\le s\le d/2$, and $(p,q)$ be any dual Strichartz exponent except for the endpoint, namely
\EQ{
  1\le p<2,\pq 1< q\le 2,\pq \frac{2}{p} + \frac{d}{q} = 2 + \frac{d}{2}.}
Let $(p_1,q_1)$ and $(p_2,q_2)$ satisfy 
\EQ{
 \pt \frac{1}{p_1}+\frac{1}{p_2}=\frac{1}{p},
 \pq \frac{1}{q_1}+\frac{1}{q_2}=\frac{1}{q} + \frac{1}{q(s)},
 \pq \frac{1}{q(s)}:=\frac{1}{2}-\frac{s}{d}, 
 \pr p\le p_1,p_2\le\I, \pq q\le q_1,q_2\le\I.}
Then for any bilinear Fourier multiplier $B$ we have
\EQ{
 \pt\left\|\int e^{itH}B[u(t),v(t)]dt\right\|_{L^2_x}
 \pn\lec \|B\|_{\L^\I_{\x}\dot B^s_{2,1,\y}+\L^\I_{\x}\dot B^s_{2,1,\z}}
  \|u\|_{L^{p_1}_tL^{q_1}_x}\|v\|_{L^{p_2}_tL^{q_2}_x}.}
where the first norm of $B$ is in the $(\x,\y)$ coordinates and the second in $(\x,\z)=(\x,\x-\y)$. 
\end{lemma}
\begin{remark}
$q(s)$ is the Sobolev exponent for the embedding $\dot B^s_{2,1}\subset L^{q(s)}$, 
and $1/q(s)$ gives the precise loss compared with the H\"older inequality. 
\end{remark}
The above estimate is a sort of composition of the Strichartz and the bilinear Fourier multiplier. Indeed, by choosing $u(t,x)=f_n(t)\fy(x)$ and $v(t,x)=\psi(x)$ with $\|f_n\|_{L^1}\le 1$ and $f_n\to \de \in \S'(\R)$, we can deduce 
\begin{corollary} \label{cor:sbil}
Let $0\le s\le d/2$, and let $(q_1,q_2)$ satisfy 
\EQ{
 \pt \frac{1}{q_1}+\frac{1}{q_2}=\frac{1}{2} + \frac{1}{q(s)},
 \pq 2\le q_1,q_2\le q(s),
 \pq \frac{1}{q(s)}:=\frac{1}{2}-\frac{s}{d}.}
Then for any bilinear Fourier multiplier $B$ we have
\EQ{
 \|B[\fy,\psi]\|_{L^2_x}
 \lec \|B\|_{\L^\I_{\x}\dot B^s_{2,1,\y}+\L^\I_{\x}\dot B^s_{2,1,\z}}\|\fy\|_{L^{q_1}_x}\|\psi\|_{L^{q_2}_x}.}
\end{corollary}
Practically, the above norm on $B$ can be estimated by the interpolation
\EQ{ \label{interpol for B}
 \pt\|B\|_{\L^\I \dot B^s_{2,1}}
  \lec \|B\|_{L^\I \dot H^{s_0}}^{1-\th} \|B\|_{L^\I \dot H^{s_1}}^{\th},}
for $s=(1-\th)s_0 + \th s_1$, with $s_0\not=s_1$ and $\th\in (0,1)$. 
For intermediate exponents, we can instead apply multilinear real interpolation to the above estimates, replacing $\L^\I \dot B^s_{2,1}$ by $L^\I \dot B^s_{2,\I}$ for free.
But this does not work in the boundary cases, and indeed the estimates would become false, for example when $s=0$ or $s=d/2$, which will be used later. 
We abbreviate those norms by
\EQ{ \label{def HB}
 \pt [H^s] := L^\I_\x \dot H^s_\y + L^\I_\x \dot H^s_\z,
 \pq [B^s] := \L^\I_\x \dot B^s_{2,1,\y} + \L^\I_\x \dot B^s_{2,1,\z}.}

The proof given below essentially contains one for the corollary, 
but it seems difficult to decouple the lemma into the Strichartz and a time independent estimate. 
On the other hand, the corollary just barely fails to be sufficient for our application. In fact, more use of the Strichartz simply causes more loss in decay, compared with the weighted decay estimate. 
That is why we do not pursue the endpoint Strichartz, which corresponds to the worst case. 
Nevertheless, we prefer the flexibility in the space-time H\"older exponents provided by Strichartz. 

\begin{proof}[Proof of Lemma \ref{sbil}] 
By symmetry between $\y$ and $\z$, it suffices to prove for the norm in the $(\x,\y)$ coordinates.  
In the case $s=0$, we get the above just by Plancherel.  
For $s>0$, we want to integrate by parts in the Fourier space. 
Since the Fourier transform of $u$ and $v$ themselves are not smooth, we will use the cubic decomposition in $x$. 
The Fourier transform of the bilinear operator is given by
\EQ{
 X := \int B f(\x,\y) d\y, 
 \pq f(\x,\y) \pt:= \int e^{itH(\x)} \ti u(t,\y)\ti v(t,\x-\y) dt 
 \pr= \F_x^\x \F_y^\y \int e^{itH_x}[u(t,x+y)v(t,x)] dt.}
First we dyadically decompose $B$ in the frequency for $\y$:
\EQ{
 \pt X = \sum_{n\in 2^\Z} X_n,
 \pq X_n := \int B_n f d\y,
 \pq B_n := \chi^n(\na)_\y B(\y,\x-\y).}
By the definition of our norm on $B$, it suffices to bound each $X_n$ uniformly. 
Next, we decompose $u(t,x)$ and $v(t,x)$ into cubes of size $n$ in $x$. 
Fix $c\in C^\I_0(\R^d)$ such that for all $x\in\R^d$ we have  
\EQ{
 \sum_{j\in \Z^d} c(x-j) = 1.}
Then for any $\fy\in\S'(\R^d)$ we have the following decomposition 
\EQ{
 \pt \ti\fy(\x) = \sum_{j\in\Z^d} e^{-inj\x}\hat\fy_j^n(\x),
 \pq \fy_j^n(x):=c(x/n-j)\fy(x),
 \pr \hat\fy_j^n(\x):=e^{inj\x}\ti\fy_j^n(\x)
 = \F_x^\x [c(x/n)\fy(x+nj)],}
where $\fy_j^n$ is supported on $\{|x/n-j|\lec 1\}$. 
Applying this decomposition to both $u$ and $v$, we get
\EQ{
 \pt X_n = \sum_{j,k\in\Z^d} \int B_n f_{j,k}^n(\x,\y) d\y,
 \pr f_{j,k}^n := \int e^{itH(\x)} e^{-inj\y-ink(\x-\y)}\hat u_j^n(t,\y)\hat v_k^n(t,\x-\y) dt.}
If $j\not=k$, then we choose $a\in\{1,\dots,d\}$ such that $|(j-k)_a|\sim|j-k|$ and integrate the phase $e^{-in(j-k)\y}$ in the variable $\y_a$ in $K$ times for arbitrary $K\in\N$. Then we get 
\EQ{
 \int B_n f_{j,k}^n d\y = \sum_{\al+\be+\ga=K} \pt[in(j-k)_a]^{-K} \iint e^{itH(\x)} e^{-inj\y-ink(\x-\y)}
 \pmt \p_{\y_a}^\al B_n \,\p_a^\be \hat u_j^n(t,\y)\,(-\p_a)^\ga \hat v_k^n(t,\x-\y) dt d\y.}
Thus denoting 
\EQ{ \label{def P**}
 \fy_{a,j}^{\ga,n} := \F^{-1}_\x e^{-inj\x} \p_{a}^\ga \hat\fy_j^n(\x) = n^\ga[-i(x/n-j)_a]^\ga c(x/n-j)\fy(x),}
we have 
\EQ{
 \pt X_n = \sum_{a=0}^d \sum_{\al=0}^K \int [\p_{\y_a}^\al B_n] G_{a,\al}^n(\x,\y) d\y,
 \pr G_{a,\al}^n(\x,\y) := \F_x^\x \F_y^\y \int e^{itH_x} F_{a,\al}^n(t,x,y) dt,
 \pr F_{0,0}^n := \sum_{j\in\Z^d} u_j(t,x+y) v_j(t,x),\pq F_{0,\al}^n := 0 \pq(\al>0),
 \pr F_{a,\al}^n(x,y) := \sum_{j,k} \sum_{\be+\ga=K-\al} \frac{(-1)^\ga}{[in(j-k)_a]^K} u_{a,j}^{\be,n} (t,x+y) v_{a,k}^{\ga,n}(t,x),}
where in the last term, the summation for $(j,k)$ is constrained to the set  $\{(j-k)_a\sim|j-k|\}$, which should be chosen to partition $\{j\not=k\}$ by $a=1,\dots,d$. 

Applying H\"older in $(\x,\y)$, we have 
\EQ{
 \|X_n\|_{L^2_\x}
  \lec \sum_{a=0}^d \sum_{\al=0}^K \|\p_{\y_a}^\al B_n\|_{L^\I_\x L^2_\y}
   \|G_{a,\al}^n\|_{L^2_{\x,\y}.}}
To estimate the $L^2_{\x,\y}$ norm, we recall the standard proof of the (non-endpoint) Strichartz estimate from the decay of $e^{itH}$. 
After using Plancherel in $(\x,\y)$, we get 
\EQ{
 \|G_{a,\al}^n\|_{L^2_{\x,\y}}^2
 = \iiint \PX{e^{i(t-t')H_x}F_{a,\al}^n(t,x,y)}{F_{a,\al}^n(t',x,y)} dy dt dt',}
and the $L^2_x$ inner product is bounded by using the $L^p$ decay of $e^{itH}$,
\EQ{
 \lec |t-t'|^{-\s} \|F_{a,\al}^n(t,x,y)\|_{L^q_x} \|F_{a,\al}^n(t',x,y)\|_{L^q_x},}
where $1\le q\le 2$ and $\s = d(1/q-1/2)$. 
Applying the Schwarz inequality to $dy$, and Hardy-Littlewood-Sobolev to $dt'dt$, we get
\EQ{
 \|G_{a,\al}^n\|_{L^2_{\x,\y}}^2
 \lec \|F_{a,\al}^n(t,x,y)\|^2_{L^p_t L^2_y L^q_x},}
for any dual Strichartz exponent $(p,q)$, except for the endpoint, which we do not consider for simplicity. 
By the support property of $u_{a,j}^{\be,n}$ and $v_{a,k}^{\ga,n}$ \eqref{def P**}, we have, noting that $|y|\gec n$ when $\al>0$, 
\EQ{
 \|F_{a,\al}^n(t,x,y)\|_{L^p_t L^2_y L^q_x}
 \pt\lec n^{-\al}\|\min(1,|y/n|^{-K})u(t,x+y)v(t,x)\|_{L^p_t L^2_y L^q_x}.}
On the other hand, by the Fourier support property of $B_n$, we have
\EQ{
 \|\p_{\y_a}^\al B_n\|_{L^\I_\x L^2_\y}
 \lec n^{\al-s}\|B_n\|_{L^\I_\x \dot H^s_\y}.}
Putting them together, we obtain 
\EQ{
 \|X_n\|_{L^2_\x}
 \lec \|B_n\|_{L^\I_\x \dot H^s_\y} n^{-s}\|\LR{y/n}^{-K}u(x+y)v(x)\|_{L^p_t L^2_y L^q_x}}
for any $K\in\N$. 
By using Young's inequality, we have for sufficiently large $K$, 
\EQ{
 \|\LR{y/n}^{-K}\fy(x+y)\psi(x)\|_{L^2_yL^q_x}^q
 \pt=\|\LR{y/n}^{-Kq}(|\fy|^q*|\psi|^q)(y)\|_{L^{2/q}_y}
 \pr\lec \|\LR{y/n}^{-Kq}\|_{L^{d/(sq)}_y} \||\fy|^q\|_{L^{q_1/q}} \||\psi|^q\|_{L^{q_2/q}}
 \pr\sim n^{sq}\|\fy\|_{L^{q_1}}^q \|\psi\|_{L^{q_2}}^q.}
Applying this to the above at each $t$, then using H\"older in $t$, and summing for all $n\in 2^\Z$, we get the desired estimate. 
\end{proof}


\subsection{Spatial integration of phase}
First we consider the spatially non-resonant part $B_j^X$. 
More precisely, we first decompose $B_j$ dyadically by using \eqref{LP decop}, 
\EQ{ \label{decop BB}
 B_j^{a,b,c} := \chi^a(\x)\chi^b(\y)\chi^c(\z) B_j = B_j^{a,b,c,X} + B_j^{a,b,c,T},}
such that $|\x|\sim a$, $|\y|\sim b$ and $|\z|\sim c$ in the support of $B_j^{a,b,c}$. The sum over $a,b,c\in 2^\Z$ can be restricted to 
\EQ{ \label{dyprod}
 a \lec b \sim c,\pq b\lec c\sim a, \pq c\lec a\sim b.}
The smooth decomposition into $B_j^{a,b,c,X}$ and $B_j^{a,b,c,T}$ will be given in Section \ref{s:multest}. 
For each $B_j^{a,b,c,X}$, we integrate the phase in $\y$, by using the identity
\EQ{
 e^{is\Om} = \frac{\na_\y\Om}{is|\na_\y\Om|^2}\cdot\na_\y e^{is\Om}.}
Then we get terms like
\EQ{ \label{phase1}
 \pt\F^{-1}\int_0^t ds \int  e^{is\Om} \left\{\BB_1\cdot\na_\y[\Jp{\zn}(\y) \Jp{\zn}(\x-\y)] 
   + \BB_2 \Jp{\zn}(\y) \Jp{\zn}(\x-\y) \right\}d\y
 \pr= \int_0^t e^{isH}\left\{\BB_1[(J\zn)^\pm,\zn^\pm] - \BB_1[\zn^\pm, (J\zn)^\pm] + \BB_2[\zn^\pm,\zn^\pm]\right\} ds,}
where $\Om=H(\x)\mp H(\y)\mp H(\x-\y)$ and $\BB_k=\BB_{k,j}^{a,b,c}$ is defined by 
\EQ{
 \pt \BB_{1,j}^{a,b,c} := \frac{\naxy\Om\cdot\na_\y\Om}{|\na_\y\Om|^2}B_j^{a,b,c,X},
 \pq \BB_{2,j}^{a,b,c} := \na_\y\cdot\frac{\na_\y\Om\cdot \naxy\Om\cdot B_j^{a,b,c,X}}{|\na_\y\Om|^2}.}
Denoting 
\EQ{
 \pt M:=\max(a,b,c),\pq m:=\min(a,b,c),\pq l:=\min(b,c),}
we assume that if $M\ll 1$ then 
\EQ{ \label{B^X bound1}
 \pt \|\BB_1^{a,b,c}\|_{[H^{1+\e}]} \lec l^{1/2-2\e}, 
 \pq \|\BB_2^{a,b,c}\|_{[H^{1+\e}]} \lec l^{1/2-2\e}M^{-1},}
for $\e>0$ small, and if $M\gec 1$ then 
\EQ{ \label{B^X bound2}
 \pt \|\BB_1^{a,b,c}\|_{[H^{3/2+\e}]} \lec l^{1-\e}\LR{a}^{-1}+m^{-\e},
 \pq \|\BB_2^{a,b,c}\|_{[H^{3/2+\e}]} \lec l^{-\e}\LR{a}^{-1} + m^{-\e},}
for $|\e|$ small. 
Note that we have a derivative loss for the first term if $\LR{a}\ll b\sim c$. This is the reason for our new normal form (cf. the discussion in Section \ref{s:normal}). 
For the definition of those norms, see \eqref{def LB} and \eqref{def HB}. 
The above estimates will be proved in Section \ref{s:multest}. 
Then applying Lemma \ref{sbil} with some fixed $s=1+\e>1$ and $3/2$, we get 
\EQ{
 \pt\left\|\sum_{a,b,c} \int e^{isH}\{\BB_1[(J\zn)^\pm,\zn^\pm] \pm \BB_1[\zn^\pm,(J\zn)^\pm]\} ds\right\|_{H^1_x}
 \pr\lec \|\sum_{M\ll 1}m^{\e/2}\BB_1^{a,b,c}\|_{[B^{1+\e}]} 
 \pn\|U^{-\e/2}J\zn\|_{L^\I_t L^{1/2-\e/6}_x} \|U^{-1/6}\zn\|_{L^{1-\e/4}_t L^6_x}
 \prq+ \|\sum_{M\sgec 1}U(m)^{\e/2}\LR{a}\LR{b}^{-1}\LR{c}^{-1}\BB_1^{a,b,c}\|_{[B^{3/2}]}
 \pmt\|U^{-\e/2}\LR{\na}J\zn\|_{L^\I_t L^{1/2-\e/6}_x} \|U^{-1/6}\LR{\na}\zn\|_{L^{3/4+\e/4}_t L^6_x},}
where the sums are for $a,b,c$ satisfying \eqref{dyprod}, and we used the convention $L^{1/p}=L^p$. The norms on $\BB_1^{a,b,c}$ are bounded by the above assumption and interpolation as follows. For the first case $M\ll 1$, we have 
\EQ{
 \|\sum_{M\ll 1}m^{\e/2}\BB_1^{a,b,c}\|_{[H^{1+\e'}]}\lec \sum_{l\ll 1}m^{\e/2}l^{1/2-2\e'} \lec 1,} 
for $\e'>0$ small, so we can replace the norm with $[B^{1+\e}]$ by the interpolation \eqref{interpol for B} for varying $\e'$. For the second case $M\gec 1$ we have 
\EQ{
 \pt\|\sum_{M\sgec 1}U(m)^{\e/2}\LR{a}\LR{b}^{-1}\LR{c}^{-1}\BB_1^{a,b,c}\|_{[H^{3/2+\e'}]}
 \pn\lec \sum_{l\in 2^\Z} U(m)^\e \LR{l}^{-\e'-1} \lec 1,}
for $|\e'|$ small. Then we can change the norm to $[B^{3/2}]$ by the interpolation \eqref{interpol for B}. 

The norm on $J\zn$ is bounded by the Sobolev embedding 
\EQ{
 \|U^{-\e/2}\LR{\na}J\zn\|_{L^\I_t L^{1/2-\e/6}_x} \lec \|J\zn\|_{L^\I_t H^1_x},}
and the norms on $\zn$ are bounded by splitting into $|t|<1$ and $|t|>1$, 
\EQ{
 \pt \|U^{-1/6}\zn\|_{L^{1-\e/4}_t L^6_x} + \|U^{-1/6}\LR{\na} \zn\|_{L^{3/4+\e/4}_t L^6_x}
 \pr\lec \|U^{-1/6}\zn\|_{L^2_t H^{1,6}_x}
 + \|tU^{-1/3}\zn\|_{L^\I_t H^{1,6}_x}
 \lec \|\zn\|_{S\cap X}.}
When restricted for large $t>T$, the above is bounded by $T^{-\e/4}$. 

In the same way, we estimate the term with $\BB_2$ by
\EQ{
 \pt\left\|\int_0^t e^{isH}\BB_2[\zn^\pm,\zn^\pm] ds\right\|_{H^1_x}
 \pr\lec \|\sum_{a,b,c}U(l)U(M)^{1/6}\LR{a}\LR{b}^{-1}\LR{c}^{-1}\BB_2^{a,b,c}\|_{[B^{1+\e}]+[B^{3/2}]}
 \pmt
   \|U^{-1}\zn\|_{L^\I_t H^1_x} \|U^{-1/6}\zn\|_{(L^{1-\e/2}\cap L^{4/3})_t H^{1,6}_x},}
where the $L^2_x$ norm is applied to the $l$ frequency component, while the $L^6_x$ norm is applied to the other with frequency $\max(b,c)\sim M$. We use $[B^{1+\e}]$ for $M\le 1$ with $L^{1-\e/2}_t L^6_x$, and $[B^{3/2}]$ for $M\ge 1$ with $L^{4/3}_t L^6_x$. 
Then the norm on $\BB_2$ is estimated by
\EQ{
 \pt \|\sum_{M<1}\cdots\BB_2^{a,b,c}\|_{[H^{1+\e'}]}
 \lec \sum_{l\le 1} l^{3/2-2\e'}M^{-5/6} \lec 1,
 \pr \|\sum_{M\ge 1}\cdots\BB_2^{a,b,c}\|_{[H^{3/2+\e''}]}
 \lec \sum_{l\in 2^\Z} U(l) \LR{l}^{-\e''-1} \lec 1,}
for $\e'>0$ small and $|\e''|$ small. Then we can replace the norms with $[B^{1+\e}]$ and $[B^{3/2}]$ respectively, by the interpolation \eqref{interpol for B}. 

The norms on $\zn$ are treated in the same way as for $\BB_1$. 
For large $t>T$, we have additional decay factor $T^{-\e/2}$.

\subsection{Time integration of phase}
For the temporally non-resonant part $B_j^{a,b,c,T}$, we integrate the phase in $s$. Then we get the time derivative of the bilinear functions, for which we use the equation again. Thus we get terms like 
\EQ{ \label{integ in s}
 \pt \int_0^t e^{isH}\left\{\BB_3[\zn^\pm,\zn^\pm]+\BB_3[s\NN^\pm,\zn^\pm] + \BB_3[\zn^\pm,s\NN^\pm]\right\}ds
 \pr + \left[e^{isH}\BB_3[s\zn^\pm,\zn^\pm]\right]_{s=0}^t,}
where we put
\EQ{
 \pt \BB_3 = \BB_{3,j}^{a,b,c} := \frac{\na_\x\Om}{\Om}B_j^{a,b,c,T},}
with $j=3,4$. Recall $|\x|\sim a$, $|\y|\sim b$, $|\z|\sim c$, $M:=\max(a,b,c)$,  $m=\min(a,b,c)$ and $l=\min(b,c)$. We assume that 
\EQ{ \label{B^T bound}
 \|\BB_3^{a,b,c}\|_{[H^{s}]}
  \lec (\LR{M}/M)^{s}l^{3/2-s}\LR{a}^{-1},}
for $0<s<2$. See \eqref{def HB} with \eqref{def LB} for the definition of norm. 
This bound will be proved in Section \ref{s:multest}. 

For the first term in \eqref{integ in s}, we use Lemma \ref{sbil} with some fixed $s=1+\e>1$ for $M\ll 1$ and $s=3/2$ for $M\gec 1$. Thus we get 
\EQ{ 
 \pt \left\|\sum_{a,b,c}\int e^{isH} \BB_3[\zn^\pm,\zn^\pm] ds\right\|_{H^1_x}
 \pr \lec \|\sum_{a,b,c} U(l)U(M)^{1/6}\LR{a}\LR{b}^{-1}\LR{c}^{-1}\BB_3^{a,b,c}\|_{[B^{1+\e}] + [B^{3/2}]} 
 \pmt \|U^{-1/6}\LR{\na}\zn\|_{(L^{1-\e/2}\cap L^{4/3})_tL^6_x} \|U^{-1}\LR{\na}\zn\|_{L^\I_t L^2_x}.}
For the norm on $\BB_3$, we have 
\EQ{
 \pt \|\sum_{M< 1}\cdots\BB_3^{a,b,c}\|_{[H^{1+\e'}]}
  \lec \sum_{l\le 1}lM^{-1/3-2\e'} \lec 1,
 \pr \|\sum_{M\ge 1}\cdots\BB_3^{a,b,c}\|_{[H^{3/2+\e'}]}
  \lec \sum_{l\in 2^\Z}U(l)\LR{l}^{-\e'-2} \lec 1,}
for $|\e'|$ small. Then by the interpolation \eqref{interpol for B}, we can change the norms to $[B^{1+\e}]$ and $[B^{3/2}]$, respectively. 
If the time interval is restricted on $(T,\I)$, then we get additional decay factor $T^{-\e/2}$ from the $L^6_x$ decay of $\zn$. 

For the time boundary term, we use Corollary \ref{cor:sbil} with $s=1$ and $3/2$. Then 
\EQ{
 \pt\|\sum_{a,b,c}e^{itH}\BB_3[t\zn^\pm,\zn^\pm]\|_{H^1_x} 
 \pr\lec  \|\sum_{a,b,c}U(l)U(M)^{1/3}\LR{a}\LR{l}^{-1/2}\LR{M}^{-1}\BB_3^{a,b,c}\|_{[B^{1}]+[B^{3/2}]} 
 \pmt \|U^{-1/3}t\zn\|_{H^{1,6}_x} \|U^{-1}\zn\|_{H^1_x},}
where we used $H^1\subset H^{1/2,3}$ for the lower frequency $l$ term when $M\gec 1$.  
The norm on $\BB_3$ is bounded by
\EQ{
 \pt \|\sum_{M< 1}\cdots\BB_3^{a,b,c}\|_{[H^{1+\e'}]}
  \lec \sum_{l \le 1}lM^{\e'-1/6} \lec 1,
 \pr \|\sum_{M\ge 1}\cdots\BB_3^{a,b,c}\|_{[H^{3/2+\e'}]}
  \lec \sum_{l\in 2^\Z}U(l)\LR{l}^{-\e'-3/2} \lec 1,}
for $|\e'|$ small, which change into the desired norms by the interpolation. 
For large $t>T$, we can gain $T^{-\e}$ by replacing $[B^1]$ with $[B^{1+\e}]$, and $H^1_x$ with $H^{1,1/2-\e/3}_x$. 
 
To the remaining terms in \eqref{integ in s}, we apply Lemma \ref{sbil} with some fixed $s=1/2+\e>1/2$ and $s=3/2$. We obtain 
\EQ{ \label{B3 aIBP}
 \pt \left\|\sum_{a,b,c}\int e^{isH}\{\BB_3[s\NN^\pm,\zn^\pm]\pm\BB_3[\zn^\pm,s\NN^\pm]\}ds \right\|_{H^1_x}
 \pr \lec \|\sum_{a,b,c}U(l)^{1/2}U(M)\LR{a}\LR{b}^{-1}\LR{c}^{-1}\BB_3^{a,b,c}\|_{[B^{1/2+\e}]+[B^{3/2}]} 
 \pmt \|U^{-1/2}\LR{\na} t\NN\|_{(L^{1-\e/2}\cap L^2)_t L^3_x} \|U^{-1}\LR{\na}\zn\|_{L^\I_t L^2_x}.}
For the norm on $\BB_3$, we estimate
\EQ{
 \pt\|\sum_{M<1}\cdots\BB_3^{a,b,c}\|_{[H^{1/2+\e'}]}
  \lec \sum_{l \le 1} l^{1/2}  \lec 1,
 \pr\|\sum_{M\ge 1}\cdots\BB_3^{a,b,c}\|_{[H^{3/2+\e'}]}
  \lec \sum_{l\in 2^\Z}U(l) \LR{l}^{-\e'-2} \lec 1,}
for $|\e'|$ small, which yields the bound on the desired norm via the interpolation. 

It remains to estimate the above norm on $\NN$. For the bilinear part of $\NN$, we use \eqref{bil nonsing}, splitting into $|t|<1$ and $|t|>1$, 
\EQ{
 \pt \|U^{-1/2}B_j[v^\pm,v^\pm]\|_{L^2_t H^{1,3}_x}
  \lec \|v\|_{L^2_t H^{1,6}_x} \|v\|_{L^\I_t L^6_x}, 
 \pr \|t^2U^{-1/2}B_j[v^\pm,v^\pm]\|_{L^\I_t H^{1,3}_x}
  \lec \|tv\|_{L^\I_t H^{1,6}_x}^2,}
for $j=3,4$. The cubic part of $\NN$ is bounded by using 
\EQ{
 \pt \|tC_j[v^\pm,v^\pm,v^\pm]\|_{(L^1\cap L^2)_t (H^1\cap H^{1,3})_x}
 \pn \lec \|tv\|_{L^\I H^{1,6}}  \|U^{-1}v\|_{L^\I H^1 \cap L^2 H^{1,6}}^2,}
for $j=1,2,3,4$. Then we apply Sobolev $U^{-1/2}(H^1\cap H^{1,3})_x\subset H^{1,3}_x$ at each $t$. 
For the quartic term we have
\EQ{
 \|Q_1(u)\|_{(H^1\cap H^{1,3})_x} 
  \lec \|u\|_{L^6_x}^4 \lec \LR{t}^{-2}\|v\|_X^4,}
which gives a better bound than the above. 
Thus we get the desired bound on $\NN$, completing the estimate \eqref{B3 aIBP}. If we restrict the time interval into $(T,\I)$, then we get additional decay $T^{-\e/2}$, where the worst decay comes from the bilinear part. 

\section{Estimates on the bilinear multipliers with divisors} \label{s:multest}
In this section we show that we can decompose the bilinear multipliers $B_j$, $j=3,4$ such that \eqref{B^X bound1}, \eqref{B^X bound2} and \eqref{B^T bound} hold, for all combinations of $\zn^\pm$, i.e., $B_j[\zn,\zn]$, $B_j[\bar{\zn},\zn]$ and $B_j[\bar{\zn},\bar{\zn}]$. 
We will use elementary geometry in the Fourier space in a way similar
to \cite{vac2}. Recall our notation in \eqref{def UHetc}. 
First we have for $|\x|\ge |\y|$  
\EQ{ \label{first bound}
 |\na H(\x)-\na H(\y)| &\sim |H'(\x)-H'(\y)| + H'(\y)|\hat{\x}-\hat{\y}|
 \pr\sim \frac{|\x|}{\LR{\x}}||\x|-|\y|| + \LR{\y}|\hat{\x}-\hat{\y}|,}
where $H'$ denotes the radial derivative of $H(\x)$. 
For the higher derivatives, we have
\EQ{
 |\na^k H(\x)| \lec \frac{\LR{\x}}{|\x|^{k-1}},}
for any $k\in\N$, hence by using the Taylor's formula when $|\x|\sim|\y|$, we obtain  
\EQ{  \label{higher bound}
 \pt|\na^k H(\x) - \na^k H(\y)|
  \lec \frac{\LR{\x}|\x-\y|}{|\x||\y|^{k-1}},}
for $|\x|\ge|\y|$. 
\subsection{Multiplier estimates for $\bar{\zn}\zn$}
In this subsection, we deal with the bilinear terms of the form
\EQ{
 e^{itH}B_j[\bar{\zn},\zn] = \F^{-1} \int e^{it\Om} B_j(\y,\x-\y) \J{\zn}^-(\y) \J{\zn}^+(\x-\y) d\y,}
with $j=3,4$. The phase $\Om$ and its derivatives are given by 
\EQ{
 \pt \Om := H(\x) + H(\y) - H(\x-\y),
 \pr \naxy\Om = \na H(\x) - \na H(\x-\y),
 \pq \na_\y\Om = \na H(\y) + \na H(\x-\y).}
Recall our notation for the dyadic component 
\EQ{
 \pt |\x|\sim a,\pq |\y|\sim b,\pq |\z|\sim c,
 \pr M=\max(a,b,c), \pq m=\min(a,b,c), \pq l=\min(b,c).}
We also denote in this subsection 
\EQ{
 \pt \al := |\hat{\z}-\hat{\x}|,
 \pq \be := |\hat{\z}+\hat{\y}|,
 \pq \y^\perp := -\hat\x\times\hat\x\times\y.}
We decompose the $(\x,\y,\z)$ region ($\x=\y+\z$) into the following five cases smoothly and exclusively. Namely, we exclude the earlier cases from the later cases. \begin{enumerate}
\item $|\y|\sim|\x|\gg |\z|$; Temporally non-resonant.
\item $\al>\sqrt{3}$; Temporally non-resonant.
\item $|\z|\gec 1$; Spatially non-resonant. 
\item $|\y^\perp|\ll M|\y|$; Temporally non-resonant.
\item Otherwise, spatially non-resonant. 
\end{enumerate}
Now we confirm the desired estimates in each case. 
\subsubsection{$|\y|\sim|\x| \gg |\z|$, into $B^T$} 
We put those dyadic pieces with $\min(b,a)\ge 8c$ into $B_j^T$ in \eqref{ST reson decop}. 
In this case we need to show \eqref{B^T bound}. 
Since $\Om=H(\x)+H(\y)-H(\z)$ and $m=|\z|\ll M$, we have 
\EQ{
 |\Om| = \Om \gec H(M) \sim M\Br{M}.}
We choose the $(\x,\z)$ coordinates to use the smallness of $\z$ region. Then  
\EQ{
 \na_\z^k\Om = -\na^k H(\z) + \na^k H(-\y),
 \pq \na_\z^k\naxy\Om = -\na^{1+k} H(\z),}
with trivial bounds
\EQ{
 \pt |\na_\z\Om| \lec \LR{M},
 \pq |\na_\z^2\Om| \lec \LR{m}/m, 
 \pr |\naxy\Om| \lec \LR{M},
 \pq |\na_\z^k\naxy\Om| \lec \LR{m}/m^k.}
Thus in this dyadic region we obtain
\EQ{
 \left\|\frac{\naxy\Om}{\Om}B_j^{a,b,c}\right\|_{L^\I_\x \dot H^s_\z}
 \lec \frac{\LR{M}}{M\LR{M}}\frac{m^{3/2}}{m^s}\frac{M}{\LR{M}} = m^{3/2-s}\LR{M}^{-1}}
for $j=3,4$ and $s=0,1,2$, which implies \eqref{B^T bound} by interpolation.  

In the remaining cases, we use the coordinates $(\x,\y)$ since $M\sim|\z|$ and $l\sim|\y|$. 

\subsubsection{$\al>\sqrt{3}$, into $B^T$} More precisely, we cut-off the multipliers by
\EQ{ \label{cut al}
 \chi_{[\al]} := \Gamma(\hat\x-\hat\z),}
for a fixed $\Gamma \in C^\I(\R^3)$ satisfying $\Gamma(x)=1$ for $|x|\ge \sqrt{3}$ and $\Gamma(x)=0$ for $|x|\le 3/2$. 
Thus we have $\al>3/2$ in this case, and $\al<\sqrt{3}$ in the remaining case. 
By exclusion of the previous case, we have $M\sim|\z|$ also. Hence we have 
\EQ{
 |\na_\y^k \chi_{[\al]}| \lec c^{-k} \sim M^{-k},}
for any $k\in\N$, which are acceptable errors in all of the following cases. 

Now we show \eqref{B^T bound} in this region. 
Since $\al>3/2$, we have $|\y|-|\z|\gec |\x|$. Hence $m\sim|\x|$ and  
\EQ{
 |\Om|=\Om \ge H(\y)-H(\z) \gec \LR{M}|\x| \sim \LR{M}m,}
and from \eqref{first bound},
\EQ{
 \pt |\naxy\Om| \lec \frac{M}{\LR{M}}|\y| + \LR{\x}\al, 
 \pq |\na_\y\Om| \lec \frac{M}{\LR{M}}|\x| + \LR{\y}\be.}
Since $\al>3/2$, we have $\be<1/2$ and so by the sine theorem
$\be \sim a\al/b \sim m/M$. 
Thus we obtain
\EQ{
 \pt |\naxy\Om| \lec \frac{M^2}{\LR{M}} + \LR{m} \lec \LR{M},
 \pq |\na_\y\Om| \lec \frac{Mm}{\LR{M}} + \frac{\LR{M}m}{M} \lec \frac{|\Om|}{M},
 \pr |\na_\y^k\naxy\Om| = |\na^{1+k}H(\z)| \lec \frac{\LR{M}}{M^{k}}.}
By \eqref{higher bound}, we have also 
\EQ{
 |\na_\y^2\Om| \lec \frac{\LR{M}m}{M^2} \lec \frac{|\Om|}{M^2}.}
Hence we have in this region 
\EQ{
 \left\|\frac{\naxy\Om}{\Om}\chi_{[\al]}B_j^{a,b,c}\right\|_{\dot H^s_\y}
  \lec \frac{\LR{M}}{\LR{M}m}\frac{M^{3/2}}{M^s} \frac{m}{\LR{m}} = \frac{M^{3/2-s}}{\LR{m}}\sim\frac{l^{3/2-s}}{\LR{a}},}
for $j=3,4$, and for $s=0,1,2$. This yields \eqref{B^T bound} by interpolation. 

\subsubsection{$|\z|\gec 1$, into $B^X$}
We put those remaining parts with $c\ge 1$ into $B^X$. 
By exclusion of the previous cases, we have 
$M\sim|\z|\gec 1$ and $\al<\sqrt{3}$. We want to show \eqref{B^X bound2} in this region. 
For the first derivatives, we have from \eqref{first bound},
\EQ{
 \pt |\naxy\Om| \sim ||\z|-|\x|| + \LR{\x}\al,
 \pq |\na_\y\Om| \sim ||\z|-|\y|| + \LR{\y}\be.}
On the other hand, we have 
\EQ{ \label{cos exp}
 \pt|\x|^2 = ||\z|-|\y||^2 + 2|\z||\y|(1+\hat{\z}\cdot\hat{\y})
  \sim ||\z|-|\y||^2 + |\z||\y|\be^2,
 \pr|\y|^2 = ||\z|-|\x||^2 + 2|\z||\x|(1-\hat{\z}\cdot\hat{\x})
  \sim ||\z|-|\x||^2 + |\z||\x|\al^2.}
Hence we have
\EQ{
 |\naxy\Om|\lec |\y|, \pq |\na_\y\Om| \sim ||\z|-|\y||+\LR{\z}\be \gec |\x|.}
For the higher derivatives, we have from \eqref{higher bound}, 
\EQ{
 \pt |\na_\y^k\naxy\Om| \lec \LR{\z}|\z|^{-k} \sim M^{1-k},
 \pq |\na_\y^k\Om| \lec |\x||\y|^{1-k} \lec |\na_\y\Om||\y|^{1-k}.}
Therefore we have in this region 
\EQ{
 \pt \left\|\frac{\naxy\Om\na_\y\Om}{|\na_\y\Om|^2}\chi_{[\al]}^C B_j\right\|_{\dot H^s_\y}
   \lec \frac{b^{1+3/2}}{ab^s}U(a)
   = l^{5/2-s}\LR{a}^{-1},
 \pr \left\|\frac{\na_\y\naxy\Om \na_\y\Om}{|\na_\y\Om|^2}\chi_{[\al]}^C B_j\right\|_{\dot H^s_\y}
   + \left\|\frac{\naxy\Om(\na_\y\Om)^2\na_\y^2\Om}{|\na_\y\Om|^4}\chi_{[\al]}^C B_j\right\|_{\dot H^s_\y}
 \pr + \left\|\frac{\naxy\Om\na_\y\Om}{|\na_\y\Om|^2}\na_\y[\chi_{[\al]}^C B_j]\right\|_{\dot H^s_\y}
   \lec \frac{b^{3/2}}{ab^s}U(a)
   = l^{3/2-s}\LR{a}^{-1},}
for $j=3,4$ and $s=0,1,2$, where $\chi_{[\al]}^C=1-\chi_{[\al]}$ is excluding the previous case. 
The above estimates imply \eqref{B^X bound2} by interpolation. 

\subsubsection{$|\y^\perp|\ll M|\y|$, into $B^T$}
More precisely, we cut-off the multipliers by
\EQ{ \label{cut perp}
 \chi_{[\perp]}:= \chi(\LR{M}|\y\times\hat{\x}|/(100Mb)),}
with $\chi \in C^\I_0(\R)$ satisfying $\chi(\mu)=1$ for $|\mu|\le 1$ and $\chi(\mu)=0$ for $|\mu|\ge 2$. Then we have
\EQ{
 |\na_\y^k \chi_{[\perp]}| \lec (\LR{M}/Mb)^{k},}
which is supported around $|\y^\perp|\sim Mb/\LR{M}$ for all $k\ge 1$. 
We have included $\LR{M}$ for later reuse, but in this subsection we can ignore it because $M\ll 1$. More precisely, we have by exclusion of the previous cases, 
\EQ{
 1\gg M\sim |\z|, \pq \al<\sqrt{3}.}
Now we prove \eqref{B^T bound} in this region. 
First if $M\not=|\z|$, then $\Om\ge H(m)\gec m$ and  
\EQ{
 \pt |\na_\y^k\naxy\Om| \lec \frac{|\y|^{1-k}}{M}, \pq |\na_\y^{k+1}\Om| \lec \frac{|\x|}{M|\y|^k},}
for $k\ge 0$, by \eqref{higher bound}. Hence we have in this region 
\EQ{
 \left\|\frac{\naxy\Om}{\Om}\chi_{[\perp]}\chi_{[\al]}^C B_j\right\|_{\dot H^s_\y}
 \lec \frac{b^{1+3/2}}{mM(Mb)^s}a
 \sim l^{3/2}(Ml)^{-s},}
for $j=3,4$ and $s=0,1,2$, where $\chi_{[\al]}^C$ is due to exclusion of the second case. Note that in this case the smallest divisor $(Mb)^s$ is from the cut-off $\chi_{[\perp]}$. The above estimate implies \eqref{B^T bound} by interpolation. 

Next we consider the main case $M=|\z|$. Let 
\EQ{
 \la := |\x|+|\y|-|\z|.}
Then by elementary geometry, we have
\EQ{
 \pt |\y^\perp| \sim |\y|(\al+\be),
 \pq \la \sim m(\al^2+\be^2) \sim m(|\y^\perp|/|\y|)^2.}
Hence the assumption in this case implies that $\la\ll M^2m$. 
Now we use the small but non-zero curvature of $H(\x)$ in the radial direction around $\x\sim 0$, to get a lower bound on $\Om$ (cf. \cite[(4.52)]{vac2}): 
\EQ{ \label{est degH''}
 \pt -\Om = [H(|\x|+|\y|)-H(\x)-H(\y)] + [H(\x-\y)-H(|\x|+|\y|)],
 \pr |H(\x-\y)-H(|\x|+|\y|)| \lec \LR{M}\la,
 \pr H(|\x|+|\y|)-H(\x)-H(\y) \sim \frac{|\x||\y|(|\x|+|\y|)}{\LR{\x}+\LR{\y}}.}
Thus we obtain
\EQ{
 |\Om| \sim M^2 m.}
For the first derivatives we have
\EQ{
 \pt |\naxy\Om| \lec M|\y| + \al \ll |\y|,
 \pq |\na_\y\Om| \lec M|\x| + \be \ll |\x|,}
since 
\EQ{
 |\x|\al \sim |\y|\be \sim m(\al+\be) \ll mM \sim |\x||\y|.} 
For the higher derivatives we have from \eqref{higher bound},
\EQ{ \label{hider1}
 |\na_\y^k\naxy\Om| \lec M^{-k},
 \pq |\na_\y^{1+k}\Om| \lec \frac{|\x|}{M|\y|^k}.}
Using the volume bound $M^2b^3$ as well, we obtain in this case 
\EQ{
 \left\|\frac{\naxy\Om}{\Om}\chi_{[\perp]}\chi_{[\al]}^C B_j\right\|_{\dot H^s_\y}
 \lec \frac{b Mb^{3/2}}{M^2m(Mb)^s}a
 \sim l^{3/2}(Ml)^{-s},}
for $j=3,4$ and $s=0,1,2$, which implies \eqref{B^T bound} by interpolation.    

\subsubsection{Otherwise, into $B^X$}
By excluding all of the previous cases, we have now
\EQ{
 1\gg M\sim|\z|,\pq \al<\sqrt{3}, \pq |\y^\perp|\gec M|\y|.}
We want to prove \eqref{B^X bound1} in this region. 
For the first derivatives, we have from \eqref{first bound},
\EQ{
 |\na_\y\Om| \sim M||\z|-|\y||+\be,
 \pq |\naxy\Om| \sim M||\z|-|\x||+\al.}
For the angular part, we have by the sine theorem
\EQ{
 \al/\be \lec |\y|/|\x|,}
since $\al$ is away from $2$. For the radial part, we have from \eqref{cos exp}
\EQ{
 M||\z|-|\x||\lec M|\y|, \pq |\na_\y \Om| \gec M|\x|.}
Thus we obtain 
\EQ{
 |\naxy\Om|/|\na_\y\Om| \lec |\y|/|\x|.}
Using \eqref{hider1} for the higher derivatives, we get 
\EQ{
 \pt \frac{|\na_\y^k\Om|}{|\na_\y\Om|} \lec \frac{|\x|}{M|\y|^{k-1}\be}
  \lec \frac{1}{|\y|^{k-2}|\y^\perp|},
 \pq \frac{|\na_\y^k\naxy\Om|}{|\na_\y\Om|} \lec \frac{1}{M^k\be}
  \lec \frac{|\y|}{M^{k-1}|\x||\y^\perp|},}
where we used 
\EQ{
 |\x||\y^\perp| \lec |\z||\y|\be.}
Now we apply the dyadic decomposition to 
\EQ{
 |\y^\perp|\sim \mu\in\{k\in 2^\Z\mid k\ge Mb\}.} 
Getting a factor $\mu|\y|^{1/2}$ from the $L^2_\y$ volume of each dyadic piece, we have 
\EQ{
 \pt \left\|\frac{\naxy\Om\na_\y\Om}{|\na_\y\Om|^2}\chi_{[\perp]}^C\chi_{[\al]}^C B_j\right\|_{\dot H^s_\y}
   \lec \sum_{\mu \ge Mb} \frac{b\mu b^{1/2}}{a \mu^{s}}a
   \lec l^{3/2}(Ml)^{1-s},
 \pr \left\|\frac{\na_\y\naxy\Om \na_\y\Om}{|\na_\y\Om|^2}\chi_{[\perp]}^C\chi_{[\al]}^C B_j\right\|_{\dot H^s_\y}
  + \left\|\frac{\naxy\Om(\na_\y\Om)^2\na_\y^2\Om}{|\na_\y\Om|^4}\chi_{[\perp]}^C\chi_{[\al]}^C B_j\right\|_{\dot H^s_\y}
 \pr + \left\|\frac{\naxy\Om\na_\y\Om}{|\na_\y\Om|^2}\na_\y[\chi_{[\perp]}^C\chi_{[\al]}^C B_j]\right\|_{\dot H^s_\y} \lec \sum_{\mu \ge Mb} \frac{b\mu b^{1/2}}{a\mu^{1+s}}a 
   \lec l^{3/2}(Ml)^{-s},}
for $j=3,4$ and $1<s\le 2$, by using interpolation for each dyadic piece in $\mu$. The condition $s>1$ is only to have convergence of the sum in $\mu$. 
As before, $\chi_{[\al]}^C=1-\chi_{[\al]}$ and $\chi_{[\perp]}^C=1-\chi_{[\perp]}$ are excluding the previous cases. 
The above estimates imply \eqref{B^X bound1}.

\subsection{Multiplier estimates for $\zn\zn$} \label{ss:zz}
We consider the bilinear terms of the form
\EQ{
 e^{itH}B_j[\zn,\zn] = \F^{-1}\int e^{it\Om}B_j(\y,\x-\y) \J{\zn}^+(\y) \J{\zn}^+(\x-\y) d\y,}
with $j=3,4$, where the phase is given by
\EQ{
 \pt \Om = H(\x) - H(\y) - H(\x-\y),
 \pr \naxy\Om = \na H(\x) - \na H(\x-\y),
 \pq \na_\y\Om = - \na H(\y) + \na H(\x-\y).}
Recall that $|\x|\sim a$, $|\y|\sim b$, $|\z|\sim c$, $M=\max(a,b,c)$, $m=\min(a,b,c)$ and $l=\min(b,c)$. 
By symmetry, we may assume without loss of generality that
\EQ{
 b \le c,}
in this subsection. We also denote in this subsection 
\EQ{
 \pt \al := |\hat{\x}-\hat{\z}|,
 \pq \be' := |\hat{\y}-\hat{\z}|,
 \pq \ga := |\hat{\x}-\hat{\y}|, 
 \pq \y^\perp := -\hat\x\times\hat\x\times\y.}
We decompose the $(\x,\y,\z)$ region into the following four cases exclusively. 
\begin{enumerate}
\item $|\y|\sim|\z|\gg|\x|$; Temporally non-resonant.
\item $\al>\sqrt{3}$; Temporally non-resonant. 
\item $|\y^\perp|\ll M|\y|/\LR{M}$; Temporally non-resonant.
\item Otherwise, spatially non-resonant. 
\end{enumerate}

\subsubsection{$|\y|\sim|\z|\gg|\x|$, into $B^T$}
We put the dyadic pieces with $\min(b,c)\ge 4a$ into $B^T$. 
Then we have $M\sim |\z|\sim|\y|\gg|\x|= m$ and  
\EQ{
 \pt |\Om| \ge H(\y) \sim M\LR{M},
 \pq |\na_\y^k\naxy\Om| \lec \LR{M}M^{-k}, \pq |\na_\y\Om| \lec \LR{M}M^{-k},}
such that in this case
\EQ{
 \left\|\frac{\naxy\Om}{\Om}B_j^{a,b,c}\right\|_{\dot H^s_\y}
  \lec \frac{\LR{M}}{M\LR{M}}\frac{M^{3/2}}{M^s}\frac{m}{\LR{m}}
  \sim l^{3/2-s}\frac{m}{M}\LR{a}^{-1},}
for $j=3,4$ and $s=0,1,2$. By interpolation, this is better than \eqref{B^T bound}. 

\subsubsection{$\al>\sqrt{3}$, into $B^T$}
We use the same cut-off $\chi_{[\al]}$ as defined in \eqref{cut al}, 
so that we have $\al>3/2$ in this case, and $\al<\sqrt{3}$ in the remaining case. Then in this region we have $M=|\y|\sim|\z|\sim|\x|$, and 
\EQ{
 |\Om| \ge H(\z) \gec M\LR{M},}
so that we end up with the same bound as in the previous case. 

In the remaining cases, $\al<\sqrt{3}$ and $|\y|\lec|\z|$, hence we have
\EQ{
 \al \lec \ga \sim \be'.}
\subsubsection{$|\y^\perp|\ll M|\y|/\LR{M}$, into $B^T$}
More precisely, we decompose the multipliers by the cut-off $\chi_{[\perp]}$ defined in \eqref{cut perp}. By exclusion of the previous cases, we have now
\EQ{
 M\sim|\x|\sim|\z|\gec|\y|\sim m.}
For the higher derivatives, we have from \eqref{higher bound}
\EQ{ \label{hider2}
 \pt |\na_\y^k\naxy\Om| \lec \frac{\LR{M}}{M^k},
 \pq |\na_\y^{1+k}\Om| \lec \frac{\LR{m}}{m^k}.}
If $|\z|-|\x|\gec|\y|$, then we have
\EQ{
 \pt |\Om|\ge H(\z)-H(\x) \gec \LR{M}m,
 \pq |\na_\x\Om| \lec \frac{\LR{M}}{M}m, \pq |\na_\y\Om| \lec \LR{M},}
and so
\EQ{
 \left\|\frac{\na_\x\Om}{\Om}\chi_{[\perp]}\chi_{[\al]}^CB_j\right\|_{\dot H^s_\y}
  \lec \frac{m^{3/2}}{M} \left(\frac{\LR{M}}{Mm}\right)^{s}  \frac{M}{\LR{M}}
  = \frac{m^{3/2}}{\LR{M}}\left(\frac{\LR{M}}{Mm}\right)^{s},}
for $j=3,4$ and $s=0,1,2$, where $\chi_{[\al]}^C=1-\chi_{[\al]}$ eliminates the previous case. By interpolation, it gives \eqref{B^T bound}. 

In the main case $|\z|-|\x|\ll|\y|$, let $\la' := |\z|+|\y|-|\x|$. 
Since $\al\lec\ga$ and both are away from $2$, we have
\EQ{
 \la' \sim m(\al^2+\ga^2) \sim m\ga^2 \sim m(|\y^\perp|/|\y|)^2 \ll \frac{M^2}{\LR{M}^2}m.}
Then by the same argument as in \eqref{est degH''}, we get 
\EQ{
 |\Om| \sim \frac{M^2}{\LR{M}}m.}
For the first derivatives, we have from \eqref{first bound},
\EQ{ \label{est first der}
 \pt |\naxy\Om| \sim \frac{M}{\LR{M}}||\z|-|\x|| + \LR{M}\al,
 \pq |\na_\y\Om| \sim \frac{M}{\LR{M}}||\z|-|\y|| + \LR{m}\be'.}
Combining them with  
\EQ{
 \al \lec \frac{m}{M}\ga, \pq \be'\sim\ga\sim\frac{|\y^\perp|}{|\y|} \ll \frac{M}{\LR{M}},}
we get
\EQ{
 |\naxy\Om| \lec m, \pq |\na_\y\Om| \lec M.}
Using the volume bound $(M/\LR{M})^2m^3$ as well, we obtain in this region
\EQ{
 \left\|\frac{\naxy\Om}{\Om}\chi_{[\perp]}\chi_{[\al]}^CB_j\right\|_{\dot H^s_\y}
  \lec \frac{\LR{M}}{M^2} \frac{Mm^{3/2}}{\LR{M}}\left(\frac{\LR{M}}{Mm}\right)^s \frac{M}{\LR{M}}
  = \frac{m^{3/2}}{\LR{M}}\left(\frac{\LR{M}}{Mm}\right)^{s},}
for $j=3,4$ and $s=0,1,2$, which gives \eqref{B^T bound} by interpolation.

\subsubsection{Otherwise, into $B^X$} \label{mest zz3}
In the remaining case, we have
\EQ{
 \pt M\sim|\x|\sim|\z|\gec|\y|\sim m,
 \pq |\y^\perp|\gec Mm/\LR{M}.}
Using $\al\lec\ga\sim\be'$ also, we have 
\EQ{
 \pt \frac{\al}{\be'}\sim\frac{\al}{\ga}\lec\frac{|\y|}{|\z|} \sim \frac{m}{M},
 \pq \frac{M}{\LR{M}}\lec \frac{|\y^\perp|}{|\y|} \lec \ga\sim\be',}
and in particular, $\ga\gec 1$ if $M\gec 1$. 
In that case, we have from \eqref{est first der} and \eqref{cos exp}, 
\EQ{
 \pt |\naxy\Om| \sim ||\z|-|\x||+M\al \sim |\y|,
 \pq |\na_\y\Om| \gec \LR{m}.}
If $M\lec 1$, we have 
\EQ{
 |\naxy\Om| \lec Mm + \frac{m}{M}\be' \sim \frac{m}{M}|\na_\y\Om|.}
Thus in both cases we have
\EQ{
 \frac{|\naxy\Om|}{|\na_\y\Om|} \lec \frac{m\LR{M}}{M\LR{m}}.}
Using \eqref{hider2} for the higher derivatives, we get 
\EQ{
 \frac{|\na_\y^{1+k}\Om|}{|\na_\y\Om|} \lec \frac{1}{m^k\be'}
  \sim \frac{1}{|\y|^k\ga} \lec \frac{1}{m^{k-1}|\y^\perp|},
 \pq \frac{|\na_\y^k\naxy\Om|}{|\na_\y\Om|}
  \lec \frac{m\LR{M}}{M^k\LR{m}|\y^\perp|}.}
Thus in this region we obtain 
\EQ{ 
 \pt \left\|\frac{\naxy\Om\na_\y\Om}{|\na_\y\Om|^2}\chi_{[\perp]}^C\chi_{[\al]}^C B_j\right\|_{\dot H^s_\y}
   \lec \sum_{\mu\gec \mu_0}\frac{m\LR{M} \mu m^{1/2}}{M\LR{m}\mu^{s}}\frac{M}{\LR{M}}
   \lec \mu_0^{1-s}\frac{m^{3/2}}{\LR{m}},
 \pr \left\|\frac{\na_\y\naxy\Om \na_\y\Om}{|\na_\y\Om|^2}\chi_{[\perp]}^C\chi_{[\al]}^C B_j\right\|_{\dot H^s_\y} 
   + \left\|\frac{\naxy\Om(\na_\y\Om)^2\na_\y^2\Om}{|\na_\y\Om|^4}\chi_{[\perp]}^C\chi_{[\al]}^C B_j\right\|_{\dot H^s_\y}
 \pr + \left\|\frac{\naxy\Om\na_\y\Om}{|\na_\y\Om|^2}\na_\y[\chi_{[\perp]}^C\chi_{[\al]}^C B_j]\right\|_{\dot H^s_\y}
   \lec \sum_{\mu\gec \mu_0}\frac{m\LR{M} \mu m^{1/2}}{M\LR{m}\mu^{1+s}}\frac{M}{\LR{M}}
   \lec \mu_0^{-s}\frac{m^{3/2}}{\LR{m}},}
for $j=3,4$ and $1<s\le 2$ by interpolation for each dyadic piece of $\mu$, where $\mu_0:=Mm/\LR{M}$, the sum for $\mu$ is on $2^\Z$, and $\chi_{[\perp]}^C$ and $\chi_{[\al]}^C$ are removing the previous regions. 
The above estimates imply \eqref{B^X bound1} and \eqref{B^X bound2}. 

\subsection{Multiplier estimates for $\bar{\zn}\bar{\zn}$} 
For the bilinear terms of the form
\EQ{
 e^{itH}B_j[\bar{\zn},\bar{\zn}] = \int e^{it\Om}B_j(\y,\x-\y) \J{\zn}^-(\y) \J{\zn}^-(\x-\y) d\y,}
we need not decompose, since it is always temporally non-resonant. 
Indeed we have
\EQ{
 \pt \Om = H(\x)+H(\y)+H(\x-\y) \gec M\LR{M}.}
By symmetry, we may assume that $|\y|\lec|\z|$. Then
\EQ{
 \pt |\naxy\Om|+|\na_\y\Om| \lec \LR{M},
 \pq |\na_\y^k\naxy\Om| \lec \frac{\LR{M}}{M^k},
 \pq |\na_\y^{1+k}\Om| \lec \frac{\LR{\y}}{|\y|^k},}
hence we get 
\EQ{
 \left\|\frac{\naxy\Om}{\Om}B_j\right\|_{\dot H^s_\y}
  \lec \frac{\LR{M}b^{3/2}}{M\LR{M}b^s} U(a)
  \lec l^{3/2-s}\LR{M}^{-1},}
for $j=3,4$ and $s=0,1,2$, which gives \eqref{B^T bound} via interpolation. 

\section{Estimates with phase derivative in cubic terms}
In this section we estimate the phase derivatives in the cubic terms $C_j$, $1\le j\le 5$. 
The higher order terms $Q_j(u)$ have been already estimated in \eqref{JN1,4}, by using regularity gain. Some parts in the frequency interactions of $C_j$ can be also estimated in a similar way, 
but if we thereby estimate the whole interactions, we can not avoid losing regularity by the factor $s\na_\x\Om$. 

On the other hand, the cubic terms have in general a large set of interactions which are resonant both in space and time. 
Hence we cannot always integrate by parts as for the bilinear terms. 

Fortunately, it turns out that those two difficulties are disjoint in the Fourier space. Namely, when the derivative loss by $\na_\x\Om$ becomes really essential, the interaction is not temporally resonant, i.e., $|\Om|$ is even bigger than $|\na_\x\Om|$. Hence we integrate on the phase in $s$ for those parts, and the other parts are bounded in the form $\na_\x\Om$. Here we do not need such precise estimates on the resulting multipliers as in the bilinear case, since the divisor appear only when $|\x|\gg 1$. 

Before applying $J$ to the cubic terms $C_j$, we apply the frequency  decomposition to each $v$ as in \eqref{LP decop}: 
\EQ{
  v = \sum_{k\in 2^\Z} v_k, \pq v_k := \chi^k(\na)v,}
Thus in the Fourier space we get terms like
\EQ{ \label{cubic with phase}
 \sum_{k_1,k_2,k_3\in 2^\Z} \int_0^t  s \na_\x\Om e^{isH} C_j[v^\pm_{k_1},v^\pm_{k_2},v^\pm_{k_3}] ds,}
where 
\EQ{
 \Om = H(\x) \mp H(\x_1) \mp H(\x_2) \mp H(\x_3),\pq
 \na_\x\Om = \na H(\x) \mp \na H(\x_1),}
and one of $v$ has to be replaced with $\zn$ if $j=5$.  
We consider the summand with $k_1\ge\max(k_2,k_3)$, choosing the smaller $\x_2,\x_3$ as integral variables. The other cases are treated in the same way. 

If $k_1\lec 1$, then all $|\na H(\x_k)|$ are bounded, hence we can bound it in $L^\I H^1$ by
\EQ{
 \|tC_j[v^\pm_{k_1},v^\pm_{k_2},v^\pm_{k_3}]\|_{L^2 L^{6/5}}
 \lec \|tv\|_{L^\I L^6} \|U^{-1}v\|_{L^\I L^2} \|U^{-1}v\|_{L^2 L^6}.}
If $1\ll k_1\sim\max(k_2,k_3)$, then we have $|\na_\x\Om|\lec\max(k_2,k_3)$. Hence the $L^\I H^1$ norm is bounded by
\EQ{ \label{C hhl}
 \|k_1^2 tC_j[v_{k_1}^\pm,v_{k_2}^\pm,v_{k_3}^\pm]\|_{L^2 L^{6/5}}
 \lec \|tv\|_{L^\I H^{1,6}} \|U^{-1}v\|_{L^\I H^1} \|U^{-1}v\|_{L^2 H^{1,6}}.}
In both cases we get $T^{-1/10}$ for $t>T$ from the $L^6_x$ decay of $U^{-1}v$. If $j=5$, we just replace one of $v$ by $\zn$. 

It remains to estimate the terms with $k_1\gg\max(1,k_2,k_3)$. 
If the sign in the phase is $\Om=H(\x)-H(\x_1)\pm\cdots$, then we have
\EQ{
 \pt \na_\x\Om = F(\x,\x_2+\x_3)\cdot(\x_2+\x_3),
 \pq F(\x,\y):=\int_0^1 \na^2 H(\x-\th\y)d\th,}
getting the factor $\x_2+\x_3$. 
In the region $|\x|\gg\LR{\y}$, we have 
\EQ{
 |\na_\x^k \na_\y^l F(\x,\y)| \lec \int_0^1 |\na^{2+k+l}H(\x-\th\y)| d\th
 \lec |\x|^{-k-l}.}
Hence the Coifman-Meyer estimate \eqref{bilCM} applies to $F$, and so we can bound the Strichartz norm as in the previous case. 

Finally, if $\Om=H(\x)+H(\x_1)\pm\cdots$ with $k_1\gg \max(1,k_2,k_3)$, we have
\EQ{
 |\Om| \gec k_1^2, \pq |\na_\x\Om| \lec k_1,}
so that we can integrate on $e^{is\Om}$ in $s$. 
Thus we get terms like
\EQ{ \label{IBP C}
 \int_0^t \frac{\na_\x\Om}{\Om} e^{isH}\left\{C[v,v,v]+sC[\Nv,v,v]\right\} ds
 + \left[\frac{\na_\x\Om}{\Om}e^{isH}sC[v,v,v]\right]_0^t,}
where we omitted indexes $j$, $k_*$ and $\pm$. Moreover if $j=5$, one of $v$ (or $\Nv$) should be replaced with $\zn$ (resp. $\NN$). 
Recall that $\Nv$ and $\NN$ are the nonlinearities defined in \eqref{eq v} and \eqref{def NN}, respectively. 
The multipliers have no singularity thanks to the frequency restrictions. 
Hence using the Coifman-Meyer estimate \eqref{bilCM}, we can bound the above term in $L^\I H^1$ by
\EQ{
  \pt\|C[v,v,v]\|_{L^2 L^{6/5}} + \|tC[\Nv,v,v]\|_{L^2 L^{6/5} + L^{4/3} L^{3/2}}
  + \|tC[v,v,v]\|_{L^\I L^2}
  \pr\lec \|v\|_{L^2 L^6}\|U^{-1}v\|_{L^\I L^3}^2 
   + (\|\Nv\|_{L^2 L^2 + L^{4/3} L^3}+\|v\|_{L^\I L^6})\|t^{1/2}U^{-1}v\|_{L^\I L^6}^2,}
where the norm on $\Nv$ is bounded by $L^2 H^{1,6/5}+L^{4/3} H^{1,3/2}$. 
For $\NN$, it was already estimated in Section \ref{ss:est w/o w}, and the same estimates \eqref{Stz on B} and \eqref{Stz on C} hold for $\Nv$ as well. 
For $t>T$, we gain $T^{-1/5}$ from the $L^6_x$ and $L^3_x$ decay of $U^{-1}v$ (or $U^{-1}\zn$). 

\section{Proof of the initial data result}
Now that we have closed estimates in $X$ for $(v,\zn)$ solving the system \eqref{eq uzn}, we can easily finish the proof for Theorem \ref{thm:init}. 
More precisely, we have obtained the following estimates: For any interval $I=(T_0,T_1)\subset\R$, and for any $(v,z)\in C(I;X)$ satisfying \eqref{eq v} and \eqref{eq uzn} on $I$, we have 
\EQ{
 \pt\|v(t)-\zn(t)\|_{X(t)} \lec \LR{t}^{-1/6}\|v(t)\|_{X(t)}^2\pq(t\in I),
 \pr\|v-\zn\|_{S(I)} \lec \LR{T_0}^{-7/10}\|v\|_{X(I)}^2, 
 \pr\|\zn-e^{-i(t-T_0)H}\zn(T_0)\|_{S(I)} \lec \LR{T_0}^{-1/2}\left\{\|v\|_{(X\cap S)(I)}^2 + \|v\|_{(X\cap S)(I)}^4 \right\}, 
 \pr\|\zn-e^{-i(t-T_0)H}\zn(T_0)\|_{X(I)} 
 \pr\lec \LR{T_0}^{-\e}\left\{\|v\|_{(X\cap S)(I)}^2 + \|v\|_{(X\cap S)(I)}^6
   + \|\zn\|_{(X\cap S)(I)}^2\right\},}
for some small $\e>0$, where the highest order $6$ comes from the partial integration replacing one of $v$ in $C_j$ by the quartic term $Q_1(u)$. 

There are at least two ways to procede: (1) construct global solutions
with the desired properties by an iteration argument, or (2) use those estimates as a priori bounds for the global solutions already known in $H^1$. 

Since the equations for the iteration sequence are a bit complicated due to the partial integration and the nonlinear transforms, we employ the latter (2) approach. 
Then we just need to ensure that the global solutions from our initial data stay in $X$ so that the a priori bounds can be rigorously applied. 
Thanks to the global wellposedness \cite{BS1} in $\psi\in C(\R;1+H^1(\R^3))$ for \eqref{GP}, it suffices to prove that only for a subset of initial data which is dense with respect to the $H^1$ norm for $u$. 

It is quite standard to see that for any $u(0)\in H^2\cap\LR{x}^{-1}H^1$, the unique global solution of \eqref{eq u} stays in the same space: since we already have it in $L^\I H^1$, we just apply the fixed point argument to the following estimates on $(0,T)$ 
\EQ{
 \pt\|u\|_{L^\I H^2} \lec \|\Nu\|_{L^2 H^{2,6/5}+L^1 H^2}
  \lec (T^{1/2}+T)\|u\|_{L^\I H^2}(1+\|u\|_{L^\I H^1})^2,
 \pr\|(x+2it\na)u\|_{H^1} \lec \|(x+2it\na)\Nu\|_{L^2 H^{1,6/5}+L^1 H^1}
 \prq\lec (T^{1/2}+T^2)(\|xu\|_{L^\I H^1}+\|u\|_{L^\I H^2})(1+\|u\|_{L^\I H^1})^2,}
where we denote the nonlinearity in the equation \eqref{eq u0} by
\EQ{
 \Nu := u^2+2|u|^2+|u|^2u.}
The above estimates follow from Strichartz applied to \eqref{eq u0}, 
and the identity $x+2it\na=e^{it\De}xe^{-it\De}$. Thus we get the solution locally in this space, and then extend it globally by Gronwall. 

To those solutions $u\in H^2\cap\LR{x}^{-1}H^1$ satisfying \eqref{init u} with $v=u_1+iUu_2$ and $\zn=v+b(u)$, we can apply the above a priori estimates on any interval, thereby getting global bounds
\EQ{
 \|v\|_{X\cap S} + \|\zn\|_{X\cap S} \lec \|v(0)\|_{X(0)} \lec \de,}
provided $\de$ is sufficiently small. By a density argument, all solutions satisfying our initial assumption can be obtained as limits of such solutions. Hence the above global bounds are transferred to them, and we get the desired scattering for them, using the identities 
\EQ{
 \pt\|e^{itH}\zn(t)-e^{iT_0H}\zn(T_0)\|_{H^1} = \|\zn(t)-e^{-i(t-T_0)H}\zn(T_0)\|_{H^1},
 \pr\left\|x\left\{e^{itH}\zn(t)-e^{iT_0H}\zn(T_0)\right\}\right\|_{H^1}
 = \left\|J\left\{\zn(t)-e^{-i(t-T_0)H}\zn(T_0)\right\}\right\|_{H^1}.}

Denote the map $W:v(0)\mapsto v_+$ thereby obtained (the inverse wave operator) locally around $0$ in $\LR{x}^{-1}H^1$. 
The continuity and the invertibility of $W$ are also proved in a standard way. The above estimates are adapted directly to differences of solutions, since after writing the system \eqref{eq uzn} for $(v,\zn)$, we did not use any genuinely nonlinear structure, but only multi-linear estimates. 
Let $(v^{(j)},\zn^{(j)})$ be small solutions of \eqref{eq uzn} obtained above, and let $f^{(j)}_T=e^{-i(t-T)H}\zn^{(j)}(T)$, for $j=1,2$ and $T\ge 0$. Then we have
\EQ{
 \pt\|\zn^{(1)}-\zn^{(2)}-(f^{(1)}_T-f^{(2)}_T)\|_{X\cap S}
 \prq\lec\sum_{j=1}^2(\|v^{(j)}\|_{X\cap S}+\|\zn^{(j)}\|_{X\cap S})
 \pn(\|v^{(1)}-v^{(2)}\|_{X\cap S}+\|\zn^{(1)}-\zn^{(2)}\|_{X\cap S}),}
and also
\EQ{
 \|v^{(1)}-v^{(2)}-(\zn^{(1)}-\zn^{(2)})\|_{X\cap S}
  \lec\sum_{j=1}^2\|v^{(j)}\|_{X\cap S}\|v^{(1)}-v^{(2)}\|_{X\cap S}.}
Hence by using the smallness assumption, we obtain
\EQ{
 \pt\|\zn^{(1)}-\zn^{(2)}\|_{X\cap S} \sim \|v^{(1)}-v^{(2)}\|_{X\cap S} \sim \|f^{(1)}_T-f^{(2)}_T\|_{X\cap S} 
 \pr\sim \|\zn^{(1)}(T)-\zn^{(2)}(T)\|_{X(T)} \sim \|v^{(1)}(T)-v^{(2)}(T)\|_{X(T)}.}
Since $W(v^{(j)}(0))=\lim_{t\to\I}e^{itH}\zn^{(j)}(t)$ in $\LR{x}^{-1}H^1$, the bi-Lipschitz property of $W$ follows from the above equivalence. 

Concerning the invertibility, or construction of the wave operator, we can argue as for Theorem \ref{thm:fin}, namely consider an approximating sequence of solutions $(v^T,\zn^T)$ with $\zn^T(T)=e^{-iTH}v_+$ for any given small $v_+\in\LR{x}^{-1}H^1$, and let $T\to\I$. Then the above a priori bound implies weak convergence of some subsequence in $X\cap S$, whose limit has the desired scattering property. \qedsymbol 
\begin{appendix}
\section{Correction for the 2D asymptotic profile in \cite{vac2}}
The asymptotic profile given in \cite{vac2} for the two dimensional case was incorrect. 
The equation for $z:=U^{-1}[u_1+(2-\De)^{-1}|u|^2]+iu_2$ was given by\footnote{This $z$ in \cite{vac2} is equal to $U^{-1}\zo$ in this paper.}
\EQ{
 i\dot z - Hz &= 2u_1^2 + 4i\na(2-\De)^{-1}U^{-1}\na(u_1\na u_2) + O(u^3)
 \pr= 2[U(z+\bar{z})/2]^2 + iH^{-1}\na[U(z+\bar{z})\cdot\na(z-\bar{z})] + O(u^3),}
so the correction term $z'$ coming from $z\bar{z}$ should be given by 
\EQ{
 z' := i\int_\I^t e^{iH(s-t)}\left\{|Uz^0|^2 + 2i\Im H^{-1}\na(U\bar{z^0}\cdot\na z^0)\right\}ds,}
where $z^0:=e^{-itH}\fy$ is the given free solution. 
The second term was missing in \cite{vac2}. 
Nevertheless the proof was correct because it was carried out assuming that all the   bilinear terms with $z^0\bar{z^0}$ had been subtracted, and those correction terms did not really affect the estimates. 
 Both terms can give non-$L^2$ contribution around $|\x|\sim 0$.

\end{appendix}

\noindent{Stephen Gustafson},  gustaf@math.ubc.ca \\
Department of Mathematics, University of British Columbia, 
Vancouver, BC V6T 1Z2, Canada

\bigskip

\noindent{Kenji Nakanishi},
n-kenji@math.kyoto-u.ac.jp \\
Department of Mathematics, Kyoto University,
Kyoto 606-8502, Japan

\bigskip

\noindent{Tai-Peng Tsai},  ttsai@math.ubc.ca \\
Department of Mathematics, University of British Columbia, 
Vancouver, BC V6T 1Z2, Canada

\end{document}